\documentclass[10pt]{article}

\usepackage[a4paper, left=2.5cm, right=2.5cm, top=3.53cm, bottom=3.53cm]{geometry}
\usepackage{setspace}
\usepackage{indentfirst}
\usepackage{abstract}
\usepackage{verbatim}
\usepackage{pifont}
\usepackage{chngcntr}
\usepackage{microtype}
\usepackage{graphicx}
\usepackage{booktabs}
\usepackage{caption}
\usepackage{enumitem}
\usepackage{calc}

\usepackage[noadjust,nocompress]{cite} 

\usepackage{tikz}
\usepackage{tikz-cd}        

\usepackage{amsmath}
\usepackage{amsthm}   

\usepackage{unicode-math}

\usepackage[colorlinks=true, linkcolor=blue, citecolor=blue, urlcolor=blue]{hyperref}

\allowdisplaybreaks[4]
\AtBeginDocument{
  \setlength{\abovedisplayskip}{6pt plus 2pt minus 2pt}
  \setlength{\belowdisplayskip}{6pt plus 2pt minus 2pt}
  \setlength{\abovedisplayshortskip}{4pt plus 1pt minus 1pt}
  \setlength{\belowdisplayshortskip}{4pt plus 1pt minus 1pt}
}

\newcommand{\C}{\mathbb{C}}
\newcommand{\Z}{\mathbb{Z}}

\numberwithin{equation}{section}

\newtheoremstyle{coddtie}
{6pt}
{6pt}
{\itshape}
{}
{\bfseries}
{}
{1em}
{\thmname{#1}\,\,\thmnumber{#2}\,\,\thmnote{#3}}

\theoremstyle{coddtie}
\newtheorem{theorem}{\textbf{Theorem}}[section]
\newtheorem{proposition}{\textbf{Proposition}}[section]
\newtheorem{lemma}{\textbf{Lemma}}[section]
\newtheorem{remark}{\textbf{Remark}}

\newtheoremstyle{coddtiex}
{10pt}
{10pt}
{}
{}
{\bfseries}
{}
{1em}
{\thmname{#1}\,\,\thmnumber{#2}\,\,\thmnote{#3}}

\theoremstyle{coddtiex}

\NewDocumentCommand{\llrr}{O{1}O{n}}{\left\lBrack  #1,#2 \right\rBrack}

\begin{document}

\title{\bfseries Simple modules over the superconformal algebra $\mathcal{S}^{\prime}(1,n)$}
\author{Jinxin Hu$^1$, Rencai L\"u$^2$ and Xinyue Wang$^3$}
\date{}
\maketitle

\footnotetext[1]{J. Hu, Department of Mathematics, Soochow University, Suzhou 215506, China,
{\em E-mail address }: \url{20244007004@stu.suda.edu.cn}}
\footnotetext[2]{R. L\"u, Department of Mathematics, Soochow University, Suzhou 215506, China,
{\em E-mail address }: \url{rlu@suda.edu.cn}}
\footnotetext[3]{X. Wang, Department of Mathematics, Soochow University, Suzhou 215506, China,
{\em E-mail address }: \url{xywang741@suda.edu.cn}}

\begin{abstract}
Let $n\geq 2$, and let $\mathcal{S}(1,n)$ be the Lie superalgebra of zero-superdivergence superderivations of
$\mathbb{C}[t^{\pm1}]\otimes\Lambda(n)$. Its derived algebra
$\mathcal{S}^\prime(1,n):=[\mathcal{S}(1,n),\mathcal{S}(1,n)]$
is well known as a superconformal algebra.
In this paper, we first study Shen-Larsson modules over $\mathcal{S}^\prime(1,n)$.
These modules, introduced by G.~Shen and T.~A.~Larsson, are constructed from modules
over the Weyl superalgebra $K_{1,n}$ and the special linear Lie superalgebra
$\mathfrak{sl}(1,n)$. We establish necessary and sufficient conditions for the simplicity of Shen-Larsson
modules and investigate their simple subquotients in the non-simple case.
Then as an application, building on the classification of simple cuspidal
$\mathcal{S}^\prime(1,n)$-modules by C. Mart\'inez, O.~Mathieu and E.~Zelmanov,
we obtain an explicit construction of all simple cuspidal modules over $\mathcal{S}^\prime(1,n)$.

\vspace*{6pt}

\end{abstract}

\section{Introduction}
We denote by $\mathbb{Z}$, $\mathbb{Z}_+$
and $\mathbb{C}$ the set of all integers, nonnegative integers and complex numbers, respectively.
The Witt algebra $W_1:=\operatorname{Der}(\mathbb{C}[t^{\pm 1}])$ and its
universal central extension, the Virasoro algebra, play a central role in both mathematics and theoretical physics.
Motivated by the profound applications of the Virasoro algebra in conformal field theory,
Neveu and Schwarz \cite{NS} and Ramond \cite{R} initiated the study of
superextensions of the Virasoro algebra, which later became known as
superconformal algebras. A systematic and formal treatment of these algebras was given
by Kac and van de Leur in \cite{KV}, where they recognized that all known superconformal
algebras coincide with infinite-dimensional Cartan type Lie superalgebras first introduced in \cite{K}.
Following \cite{KV},
a $\mathbb{Z}$-graded Lie superalgebra
$L = \bigoplus_{i\in\mathbb{Z}} L_i$ is called a superconformal algebra if

(1) $L$ is graded simple;

(2) the Witt algebra $W_1$ embeds into the even part $L_{\bar{0}}$ of $L$;

(3) the dimensions $\dim L_i$, $i\in \Z$, are uniformly bounded.

Denote by $\Lambda(n)$ the exterior algebra in $n$ odd variables $\xi_1,\xi_2,\cdots,\xi_n$.
Let $W(1,n)$ denote the Lie superalgebra of all superderivations of superalgebra $\mathbb{C}[t^{\pm 1}] \otimes \Lambda(n)$,
and let $\mathcal{S}(1,n)$ be its subalgebra consisting of superderivations with zero superdivergence.
Following Kac and van de Leur \cite{KV}, both $W(1,n)$ and $\mathcal{S}^\prime(1,n) := [\mathcal{S}(1,n),\mathcal{S}(1,n)]$ (for $n\geq 2$)
are superconformal algebras.
In particular,
$\mathcal{S}^\prime(1,2)$ is nothing else but
``$\mathrm{SU}_2$-superconformal algebra'' appearing in the physics literature
\cite{ABD,NS,CK,ET}.
It was shown in \cite{KV} that there exist non-trivial central extensions of $\mathcal{S}^\prime(1,n)$
only when $n=2$. The universal central extension of $\mathcal{S}^\prime(1,2)$ is called the small
$N = 4$ super Virasoro algebra.

The main aim of this paper is to study Shen-Larsson modules over
$\mathcal{S}^\prime(1,n)$. These modules were originally introduced by
G.~Shen and T.~A.~Larsson \cite{L,SG} for the Cartan type Lie algebra
$W_n$, which consists of all derivations of the Laurent polynomial algebra
$\mathbb{C}[t_i^{\pm1}\mid 1\leq i\leq n].$
They were later extended to the Lie superalgebra
$W(m,n):=\operatorname{Der}\big(\mathbb{C}[t_i^{\pm1}\mid 1\leq i\leq m]
\otimes \Lambda(n)\big).$
Shen-Larsson modules (called tensor modules) over $W(m,n)$ were studied in \cite{XW}. For other Cartan type Lie algebras, some results
on Shen-Larsson modules are also available; see, for example, \cite{DGYZ,FT,HL,LGW,EP}.

One of the key significances of studying Shen-Larsson modules is that, in the case of
Cartan type $W$ Lie (super)algebras, the classification problems for many classes of
simple modules eventually reduce to understanding these modules.
The classification of simple Harish-Chandra modules (weight modules
with finite-dimensional weight spaces) over Cartan type $W$ Lie algebras was obtained by Y.~Billig, V.~Futorny, D.~Grantcharov and V.~Serganova
\cite{BF1,GS2}. Their results show that many
simple Harish-Chandra modules can be realized in terms of Shen-Larsson
modules. In the super setting, the classification of simple strong Harish-Chandra
modules over Cartan type $W$ Lie superalgebras was established by Y.~Billig,
Y.~Cai, V.~Futorny, K.~Iohara, I.~Kashuba, R.~L\"u and Y.~Xue;
see \cite{CLX,XL,BFIK}. Shen-Larsson modules also play a crucial
role in these classifications.

Motivated by the aforementioned important results, we study Shen-Larsson modules over $\mathcal{S}^\prime(1,n)$.
Our construction yields a large family of simple modules over $\mathcal{S}^\prime(1,n)$,
including many simple weight modules as special cases.
We establish necessary and sufficient conditions for the simplicity of Shen-Larsson
modules and investigate their simple subquotients in the non-simple case.
This further provides a new perspective on simple cuspidal modules over $\mathcal{S}^\prime(1,n)$.
In their recent work \cite{MMZ}, C. Mart\'inez, O. Mathieu, and E. Zelmanov classified the simple cuspidal modules
over all known superconformal algebras of semisimple rank at least $1$, including $\mathcal{S}^\prime(1,n)$.
As an application, building on this classification, we give an explicit construction of all simple cuspidal modules over $\mathcal{S}^\prime(1,n)$.

The rest of this paper is organized as follows. Section 2 is devoted to recalling some basic notations and preliminary results
that will be needed in the subsequent sections.
In Section 3, we recall and investigate Shen-Larsson modules over the simple superconformal
algebra $\mathcal{S}^\prime(1,n)$.
The main results of this section are twofold: we establish necessary and sufficient conditions for the simplicity of these Shen-Larsson modules, and we
investigate their simple subquotients when they fail to be simple; see Theorems \ref{t1}, \ref{t3} and Propositions \ref{p6}, \ref{p98}, \ref{p99}.
Finally, in Section 4, we apply these results to give an explicit construction of all
simple cuspidal modules over $\mathcal{S}^\prime(1,n)$, which is presented in Theorem \ref{t100}.

\section{Preliminaries}
All the vector spaces and algebras in this paper are over $\mathbb{C}$. A super vector space $V$ is a vector space
endowed with a $\mathbb{Z}_2$-gradation $V = V_{\bar{0}}\oplus V_{\bar{1}}$.
The parity of a homogeneous element $v\in V_{i}$ is denoted by
$|v|=i\in\mathbb{Z}_2$. Throughout this paper, when we write $|v|$ for an element $v\in V$, we always assume that $v$
is a homogeneous element. 
Any module over a Lie superalgebra or an associative superalgebra is assumed to be $\mathbb{Z}_2$-graded.
Let $\mathfrak{G}$ be a Lie superalgebra or an associative superalgebra, there is a parity-change
functor $\Pi$ on the category of $\mathfrak{G}$-modules interchanging the $\Z_2$-grading of a module.
For $a,b\in\mathbb{Z}$, set $\llrr[a][b]:=\{a,a+1,\cdots,b\}$ if $a\leq b$,  otherwise, $\llrr[a][b]:=\varnothing$.

\subsection{The superconformal algebra $\mathcal{S}^\prime(1,n)$}

Assume that $I=\{i_1,i_2,\cdots,i_r\}$, $i_1<i_2<\cdots<i_r$, is a subset of $\llrr$, we set
$\xi_I:=\xi_{i_1}\xi_{i_2}\cdots\xi_{i_r}\in \Lambda(n)$. For convince, let $\xi_\varnothing:=1$.
Moreover, denote by $|I|$ the cardinality of $I\subseteq \llrr$.

The Witt superalgebra $\mathcal{W}(1,n)$ is the Lie superalgebra of super-derivations of $\mathbb{C}[t^{\pm 1}]\otimes\Lambda(n)$, that is,
\[\mathcal{W}(1,n):=\mathrm{Der}(\mathbb{C}[t^{\pm 1}]\otimes\Lambda(n))=\left\{ f_0\frac{\mathrm{d}}{\mathrm{d}t}+\sum_{i\in\llrr}^{}{f_i\frac{\partial}{\partial \xi _i}}
 \,\middle|\, f_i\in \mathbb{C}[t^{\pm 1}]\otimes\Lambda(n),i\in\llrr[0][n]\right\}.\]
$\mathcal{W}(1,n)$ has a standard basis
\[\left\{ t^m\xi _I\frac{\mathrm{d}}{\mathrm{d}t},t^m\xi _I\frac{\partial}{\partial \xi _i}
 \,\middle|\, m\in \mathbb{Z} ,I\subseteq\llrr,i\in\llrr \right\} \]
with Lie brackets given by
\[
\left[ t^m\xi _I\partial ,t^k\xi _J\partial ^{\prime} \right]
=t^m\xi _I\partial \left( t^k\xi _J \right) \partial ^{\prime}-\left( -1 \right) ^{\left( |I|+|\partial | \right)
 \left( |J|+|\partial ^{\prime}| \right)}t^k\xi _J\partial ^{\prime}\left( t^m\xi _I \right) \partial
 ,
\]
where $\partial,\partial ^{\prime}\in\left\{ \frac{\mathrm{d}}{\mathrm{d}t},\frac{\partial}{\partial \xi _1},
\frac{\partial}{\partial \xi _2},\cdots ,\frac{\partial}{\partial \xi _n} \right\} $.

For $n\geq 2$, let
$\mathcal{S}(1,n)$ be the subalgebra of the Witt superalgebra $\mathcal{W}(1,n)$ consisting of all superderivations with vanishing superdivergence.
Alternatively,
\[\mathcal{S}(1,n):=\left\{ f_0\frac{\mathrm{d}}{\mathrm{d}t}+\sum_{i\in\llrr}^{}{f_i\frac{\partial}{\partial \xi _i}}\in \mathcal{W}(1,n)
\,\middle|\, \frac{\mathrm{d}\left( f_0 \right)}{\mathrm{d}t}+\sum_{i\in\llrr}^{}{\left( -1 \right) ^{|f_i|}
\frac{\partial \left( f_i \right)}{\partial \xi _i}}=0\right\}.\]
Let
$\mathcal{S}^\prime(1,n):=[\mathcal{S}(1,n),\mathcal{S}(1,n)].$
This derived algebra is the superconformal algebra in \cite{KV}.

$\mathcal{S}^\prime(1,n)$ is spanned by the elements
\[
t^k\frac{\partial \left( \xi _I \right)}{\partial \xi _i}
\frac{\mathrm{d}}{\mathrm{d}t}+\left( -1 \right) ^{1+|I|}kt^{k-1}\xi _I\frac{\partial}{\partial \xi _i}
\quad\mathrm{and}\quad
t^k\frac{\partial \left( \xi _I \right)}{\partial \xi _i}\frac{\partial}{\partial \xi _j}
+t^k\frac{\partial \left( \xi _I \right)}{\partial \xi _j}\frac{\partial}{\partial \xi _i}
\]
for all $k\in\mathbb{Z}$, $I\subseteq \llrr$ and $i,j\in\llrr$. Moreover, we have
\[\mathcal{S}(1,n)=\mathcal{S}^\prime(1,n)\oplus\C \xi _{\llrr}\frac{\mathrm{d}}{\mathrm{d}t}.\]

\subsection{The Lie superalgebras $\mathfrak{gl}(1,n)$ and $\mathfrak{sl}(1,n)$}

The general linear Lie superalgebra $\mathfrak{gl}(1,n)$ consisting of all $(1+n)\times(1+n)$ matrices
has the $\mathbb{Z}_2$-gradation
\[\mathfrak{gl}(1,n)=\mathfrak{gl}(1,n)_{\bar{0}}\oplus\mathfrak{gl}(1,n)_{\bar{1}},\]
where the even part $\mathfrak{gl}(1,n)_{\bar{0}}=\mathfrak{gl}(1)\oplus\mathfrak{gl}(n)$ has basis
$\left\{ E_{0,0} \right\} \cup \left\{ E_{i,j} \,\middle|\, i,j\in\llrr \right\}$
and the odd part $\mathfrak{gl}(1,n)_{\bar{1}}$ has basis $\left\{ E_{0,i},E_{j,0} \,\middle|\,  i,j\in\llrr \right\}$. The Lie brackets are
given by \[\left[ E_{i,j},E_{s,t} \right] =\delta _{j,s}E_{i,t}-
\left( -1 \right) ^{\left| E_{i,j} \right|\cdot \left| E_{s,t} \right|}\delta _{t,i}E_{s,j}\,\,\,
\text{for all }i,j,s,t\in\llrr[0][n].\]
Moreover, $\mathfrak{gl}(1,n)$ has a $\mathbb{Z}$-gradation
\[\mathfrak{gl}(1,n)=\mathfrak{gl}(1,n)_{-1}\oplus\mathfrak{gl}(1,n)_{0}\oplus\mathfrak{gl}(1,n)_{1},\]
where $\mathfrak{gl}(1,n)_{-1}=\text{span}_{\mathbb{C}}\left\{ E_{i,0} \,\middle|\, i\in\llrr \right\}$,
$\mathfrak{gl}(1,n)_{1}=\text{span}_{\mathbb{C}}\left\{ E_{0,i} \,\middle|\, i\in\llrr \right\}$
and $\mathfrak{gl}(1,n)_{0}=\mathfrak{gl}(1,n)_{\bar{0}}$.

For any element $A=\sum_{i,j\in \llrr[0]}{a_{i,j}E_{i,j}}\in \mathfrak{gl}(1,n)$,
the supertrace of $A$ is defined via \[\mathrm{str}(A):=a_{0,0}-\sum_{i\in\llrr}a_{i,i}.\]
The special linear Lie superalgebra $\mathfrak{sl}(1,n)$ is defined as
\[\mathfrak{sl}(1,n):=\left\{ A\in \mathfrak{gl}(1,n)\,\middle|\,\mathrm{str}(A)=0 \right\}.\]
The $\mathbb{Z}_2$-gradation and $\mathbb{Z}$-gradation of $\mathfrak{sl}(1,n)$ are inherited from that of
$\mathfrak{gl}(1,n)$.
In the remainder of this paper, when we refer to $\mathfrak{gl}(1,n)$ and \(\mathfrak{sl}(1,n)\),
we always assume by default that \(n \geq 2\).

Let $\mathfrak{g}=\mathfrak{gl}(1,n)$ or $\mathfrak{sl}(1,n)$.
For any simple $\mathfrak{g}_{0}$-module $M$, we extend $M$ trivially to
a $\mathfrak{g}_{0}\oplus\mathfrak{g}_{-1}$-module. The induced module
$\mathrm{Ind}_{\mathfrak{g}_{0}\oplus\mathfrak{g}_{-1}}^{\mathfrak{g}}(M)$, which is called the Kac module
of $M$ and denoted by $K(M)$. Recall the following result from \cite{CV}.

\begin{lemma}[({\cite[Theorem 4.1]{CV}})]\label{l1}
For any simple $\mathfrak{g}_{0}$-module $M$, the corresponding Kac module
$K(M)$ has a unique maximal submodule, and its unique simple top is denoted by $L(M)$.
Any simple $\mathfrak{g}$-module is isomorphic to $L(M)$ for some simple $\mathfrak{g}_{0}$-module $M$ up to parity-change.
\end{lemma}

\subsection{The Weyl superalgebras $K_{m,n}$ and $K^+_{m,n}$}
The Weyl superalgebra $K_{m,n}$
is the simple associative superalgebra
\[\mathbb{C}\left[t_1^{\pm 1},\cdots, t_m^{\pm 1},\xi_1,\cdots,\xi_n, \frac{\partial}{\partial t_1},\cdots,\frac{\partial}{\partial t_m},
\frac{\partial}{\partial \xi_1},\cdots,\frac{\partial}{\partial \xi_n}\right],\]
while $K^+_{m,n}$ is the simple associative superalgebra
\[\mathbb{C}\left[t_1,\cdots, t_m,\xi_1,\cdots,\xi_n, \frac{\partial}{\partial t_1},\cdots,\frac{\partial}{\partial t_m},
\frac{\partial}{\partial \xi_1},\cdots,\frac{\partial}{\partial \xi_n}\right].\]

A weight module $V$ over $K_{m,n}$ (resp.\ $K^+_{m,n}$) is a module such that all elements in
\[\left\{ t_i\frac{\partial}{\partial t_i},\xi _j\frac{\partial}{\partial \xi _j} \,\middle|\, i\in \llrr[1][m] ,j\in\llrr \right\}\]
act  diagonalizably  on $V$. 
Then $\lambda=(\lambda_1,\cdots, \lambda_{m+n})\in \C^{m+n}$ is called a weight of $V$ if the weight space $V_{\lambda}:=\{v\in V~|~ t_i\frac{\partial}{\partial t_i} v=\lambda_i v,\xi _j\frac{\partial}{\partial \xi _j}v=\lambda_{m+j} v\}$ is nonzero. We denote by $\mathrm{supp}(V)$ the weight set of a weight module $V$.

We recall the results on simple modules of $K^+_{m,n}$, see for example \cite{FGM,XL,XW}.

\begin{theorem}\label{t2}
(1) For a given $\lambda\in\mathbb{C}^m$, up to isomorphism,
the simple weight $K^+_{m,0}$-module $V$ with $\lambda\in \mathrm{supp}(V)$ is
uniquely determined, which is
isomorphic to
\[A_{m,0}^+(\lambda):=V_1\otimes V_2\otimes\cdots\otimes V_m,\]
where $V_i=\mathbb{C}[t_i]$ (resp.\ $\mathbb{C}[t_i^{\pm1}]/\mathbb{C}[t_i]$, $t_i^{\lambda_i}\mathbb{C}[t_i^{\pm1}]$)
if $\lambda_i\in\mathbb{Z}_+$ (resp.\ $-\mathbb{Z}_{>0}$, $\mathbb{C}\setminus\mathbb{Z}$).

(2) Any simple weight $K^+_{m,n}$-module is isomorphic to $A_{m,n}^+(\lambda):=A_{m,0}^+(\lambda)\otimes \Lambda(n)$ for some $\lambda \in \C $ up to a parity-change.

(3) Any simple weight $K_{m,n}$-module is isomorphic to
\[V_1\otimes V_2\otimes\cdots\otimes V_m\otimes\Lambda(n) \] up to a parity-change,
where each $V_i$ is isomorphic to  $ t_i^{\lambda_i}\mathbb{C}[t_i^{\pm 1}]$ for some $\lambda_i\in\mathbb{C}$.

(4) Any simple $K^+_{m,n}$-module (resp.\ $K_{m,n}$-module) is isomorphic to $P_1\otimes \Lambda(n)$ up to a parity-change,
where
\[P_1:=\left\{ p\in P \,\middle|\, \frac{\partial}{\partial \xi _i}u=0\,,\,
\forall i\in\llrr
 \right\}\]
is a simple $K^+_{m,0}$-module (resp.\ $K_{m,0}$-module).
\end{theorem}

We denote by $\Gamma(\lambda)$ the weight set of the simple weight $K^+_{m,0}$-module $A_{m,0}^+(\lambda)$,
that is,
\[\Gamma(\lambda)=X_1\times X_2\times \cdots \times X_m\subseteq \lambda+\mathbb{Z}^m,\]
where $X_i=\mathbb{Z}_+$ (resp.\ $-\mathbb{Z}_{>0}$, $\mathbb{C}\setminus\mathbb{Z}$)
if $\lambda_i\in\mathbb{Z}_+$ (resp.\ $-\mathbb{Z}_{>0}$, $\mathbb{C}\setminus\mathbb{Z}$).

The general linear Lie algebra $\mathfrak{gl}(n)$ can be regarded as a Lie subalgebra of $K^+_{n,0}$
by identifying $E_{i,j}$ with $t_i\frac{\partial}{\partial t_j}$ for all $i,j\in\llrr$.
Let $\lambda\in\mathbb{C}^n$ and consider the following subspace of the simple weight $K^+_{n,0}$-module $A^+_{n,0}(\lambda)$:
\[N(\lambda):=\left\{ v\in A_{n,0}^{+}(\lambda) \,\middle|\, \left( \sum_{s=1}^n{t_s\frac{\partial}{\partial t_s}} \right) v=|\lambda |v \right\},\]
where $|\lambda|:=\lambda_1+\lambda_2+\cdots+\lambda_n$.
According to \cite[Proposition 2.12]{BBL}, $N(\lambda)$ is a simple weight module
for both
$\mathfrak{gl}(n)$ and $\mathfrak{sl}(n)$.
We recall the following known result.
\begin{lemma}[({\cite[Lemma 2.4]{XW}})]\label{l4}
  Suppose that $n\geq 2$. Let $V$ be a simple weight $\mathfrak{gl}(n)$-module satisfying that
  \[\left(E_{i,i}+E_{i,i}E_{j,j}-E_{i,j}E_{j,i}\right)V=0,\] for all $i,j\in\llrr$ with $i\ne j$. Then $V\cong N(\lambda)$
  for some $\lambda\in \mathbb{C}^n$.
\end{lemma}

\section{Shen-Larsson modules of $\mathcal{S}^\prime(1,n)$}

From \cite{XW}, we know that there is a Lie superalgebra homomorphism $\pi$ from $\mathcal{W}(1,n)$
to the tensor superalgebra $K_{1,n}\otimes U(\mathfrak{gl}(1,n))$ given by
\begin{align*}
  \pi \left( t^k\xi _I\frac{\mathrm{d}}{\mathrm{d}t} \right) ={}&t^k\xi _I\frac{\mathrm{d}}{\mathrm{d}t}\otimes 1+kt^{k-1}\xi _I\otimes E_{0,0}+\left( -1 \right) ^{\left| I \right|-1}\sum_{s\in \llrr}^{}{\frac{\partial \left( t^k\xi _I \right)}{\partial \xi _s}\otimes E_{s,0}},
\\
\pi \left( t^k\xi _I\frac{\partial}{\partial \xi _i} \right) ={}&t^k\xi _I\frac{\partial}{\partial \xi _i}\otimes 1+kt^{k-1}
\xi _I\otimes E_{0,i}+\left( -1 \right) ^{\left| I \right|-1}\sum_{s\in \llrr}^{}{\frac{\partial \left( t^k\xi _I \right)}{\partial \xi _s}\otimes E_{s,i}}
,
\end{align*}
where $k\in\mathbb{Z}$, $I\subseteq\llrr$ and $i\in\llrr$.
Clearly, $\pi$ induces an associative superalgebra homomorphism from $U(\mathcal{W}(1,n))$
to $K_{1,n}\otimes U(\mathfrak{gl}(1,n))$, which we still denote by $\pi$.

We observe that
\begin{align*}
  &\pi \left( t^k\frac{\partial \left( \xi _I \right)}{\partial \xi _i}
  \frac{\mathrm{d}}{\mathrm{d}t}+\left( -1 \right) ^{1+|I|}kt^{k-1}\xi _I\frac{\partial}{\partial \xi _i} \right)
\\
={} &t^k\frac{\partial \left( \xi _I \right)}{\partial \xi _i}\frac{\mathrm{d}}{\mathrm{d}t}\otimes 1+kt^{k-1}\frac{\partial \left( \xi _I \right)}{\partial \xi _i}\otimes \left( E_{0,0}+E_{i,i} \right) +\left( -1 \right) ^{\left| \frac{\partial \left( \xi _I \right)}{\partial \xi _i} \right|-1}\sum_{s\in \llrr\setminus \left\{ i \right\}}^{}{t^k\frac{\partial \left( \xi _I \right)}{\partial \xi _s\partial \xi _i}\otimes E_{s,0}}
\\
&+\left( -1 \right) ^{1+|I|}kt^{k-1}\xi _I\frac{\partial}{\partial \xi _i}\otimes 1
+\left( -1 \right) ^{1+|I|}k\left( k-1 \right) t^{k-2}\xi _I\otimes E_{0,i}
\\
&+k\sum_{s\in \llrr\setminus
\left\{ i \right\}}^{}{t^{k-1}\frac{\partial \left( \xi _I \right)}{\partial \xi _s}\otimes E_{s,i}},
\end{align*}
and
\begin{align*}
&\pi \left( t^k\frac{\partial \left( \xi _I \right)}{\partial \xi _i}\frac{\partial}{\partial \xi _j}
  +t^k\frac{\partial \left( \xi _I \right)}{\partial \xi _j}\frac{\partial}{\partial \xi _i} \right)
\\
={} &t^k\frac{\partial \left( \xi _I \right)}{\partial \xi _i}\frac{\partial}{\partial \xi _j}\otimes 1+kt^{k-1}\frac{\partial \left( \xi _I \right)}{\partial \xi _i}\otimes E_{0,j}+t^k\frac{\partial \left( \xi _I \right)}{\partial \xi _j}\frac{\partial}{\partial \xi _i}\otimes 1+kt^{k-1}\frac{\partial \left( \xi _I \right)}{\partial \xi _j}\otimes E_{0,i}
\\
&+\left( -1 \right) ^{\left| \frac{\partial \left( \xi _I \right)}{\partial \xi _i} \right|-1}\sum_{s\in\llrr\setminus\{i,j\}}^{}{t^k\frac{\partial \left( \xi _I \right)}{\partial \xi _s\partial \xi _i}\otimes E_{s,j}}+\left( -1 \right) ^{\left| \frac{\partial \left( \xi _I \right)}{\partial \xi _j} \right|-1}\sum_{s\in\llrr\setminus\{i,j\}}^{}{t^k\frac{\partial \left( \xi _I \right)}{\partial \xi _s\partial \xi _j}\otimes E_{s,i}}
\\
&+\left( -1 \right) ^{\left| \frac{\partial \left( \xi _I \right)}{\partial \xi _i} \right|-1}t^k\frac{\partial \left( \xi _I \right)}{\partial \xi _i\partial \xi _j}\otimes \left( E_{i,i}-E_{j,j} \right)
.
\end{align*}
This
shows that $\pi(\mathcal{S}^\prime(1,n))\subseteq K_{1,n}\otimes U(\mathfrak{sl}(1,n))$.
Let $P$ be a $K_{1,n}$-module and $M$ be an $\mathfrak{sl}(1,n)$-module. We equip the tensor space $P\otimes_{\mathbb{C}} M$ with
an $\mathcal{S}^\prime(1,n)$-module structure via
\[x\cdot (p\otimes v):=\pi(x)\cdot (p\otimes v)\,\,,\,\,\forall x\in \mathcal{S}^\prime(1,n)\,,\,p\in P\,,\, v\in M.\]
The resulting $\mathcal{S}^\prime(1,n)$-module is
called the Shen-Larsson module and
denoted by $F(P,M)$.

\subsection{Simplicity of Shen-Larsson modules over $\mathcal{S}^\prime(1,n)$}

\begin{lemma}\label{l3}
  For all $i,j\in\llrr$, we have
$
 K_{1,n}\otimes E_{0,j}E_{0,i} \subseteq \pi(U(\mathcal{S}^\prime(1,n)))
$.
\end{lemma}
\begin{proof}
  For all $I\subseteq\llrr$, $m,k\in\mathbb{Z}$ and $i,j\in\llrr$,
  we have
\begin{align*}
  &\pi \left( t^k\frac{\partial \left( \xi _I \right)}{\partial \xi _j}\frac{\mathrm{d}}{\mathrm{d}t}+\left( -1 \right) ^{1+|I|}kt^{k-1}\xi _I\frac{\partial}{\partial \xi _j} \right) \cdot \pi \left( t^{m-k}\frac{\partial}{\partial \xi _i} \right)
\\
={} &t^k\frac{\partial \left( \xi _I \right)}{\partial \xi _j}\frac{\mathrm{d}}{\mathrm{d}t}t^{m-k}\frac{\partial}{\partial \xi _i}\otimes 1+kt^{k-1}\frac{\partial \left( \xi _I \right)}{\partial \xi _j}t^{m-k}\frac{\partial}{\partial \xi _i}\otimes \left( E_{0,0}+E_{j,j} \right)
\\
&-\left( -1 \right) ^{\left| \frac{\partial \left( \xi _I \right)}{\partial \xi _j} \right|-1}\sum_{s\in \llrr\setminus \left\{ j \right\}}^{}{t^k\frac{\partial \left( \xi _I \right)}{\partial \xi _s\partial \xi _j}t^{m-k}\frac{\partial}{\partial \xi _i}\otimes E_{s,0}}+\left( -1 \right) ^{1+|I|}kt^{k-1}\xi _I\frac{\partial}{\partial \xi _j}t^{m-k}\frac{\partial}{\partial \xi _i}\otimes 1
\\
&-\left( -1 \right) ^{1+|I|}k\left( k-1 \right) t^{k-2}\xi _It^{m-k}\frac{\partial}{\partial \xi _i}\otimes E_{0,j}+k\sum_{s\in \llrr\setminus \left\{ j \right\}}^{}{t^{k-1}\frac{\partial \left( \xi _I \right)}{\partial \xi _s}t^{m-k}\frac{\partial}{\partial \xi _i}\otimes E_{s,j}}
\\
&+\left( m-k \right) t^k\frac{\partial \left( \xi _I \right)}{\partial \xi _j}\frac{\mathrm{d}}{\mathrm{d}t}t^{m-k-1}\otimes E_{0,i}+\left( m-k \right) kt^{k-1}\frac{\partial \left( \xi _I \right)}{\partial \xi _j}t^{m-k-1}\otimes \left( E_{0,0}+E_{j,j} \right) E_{0,i}
\\
&+\left( -1 \right) ^{\left| \frac{\partial \left( \xi _I \right)}{\partial \xi _j} \right|-1}\left( m-k \right) \sum_{s\in \llrr\setminus \left\{ j \right\}}^{}{t^k\frac{\partial \left( \xi _I \right)}{\partial \xi _s\partial \xi _j}t^{m-k-1}\otimes E_{s,0}E_{0,i}}
\\
&+\left( -1 \right) ^{1+|I|}\left( m-k \right) kt^{k-1}\xi _I\frac{\partial}{\partial \xi _j}t^{m-k-1}\otimes E_{0,i}
\\
&+\left( -1 \right) ^{1+|I|}\left( m-k \right) k\left( k-1 \right) t^{k-2}\xi _It^{m-k-1}\otimes E_{0,j}E_{0,i}
\\
&+\left( m-k \right) k\sum_{s\in \llrr\setminus \left\{ j \right\}}^{}{t^{k-1}\frac{\partial \left( \xi _I \right)}{\partial \xi _s}t^{m-k-1}\otimes E_{s,j}E_{0,i}}
\\
={} &t^m\frac{\partial \left( \xi _I \right)}{\partial \xi _j}\frac{\mathrm{d}}{\mathrm{d}t}\frac{\partial}{\partial \xi _i}\otimes 1+\left( m-k \right) t^{m-1}\frac{\partial \left( \xi _I \right)}{\partial \xi _j}\frac{\partial}{\partial \xi _i}\otimes 1
\\
&+kt^{m-1}\frac{\partial \left( \xi _I \right)}{\partial \xi _j}\frac{\partial}{\partial \xi _i}\otimes \left( E_{0,0}+E_{j,j} \right) +\left( -1 \right) ^{1+|I|}kt^{m-1}\xi _I\frac{\partial}{\partial \xi _j}\frac{\partial}{\partial \xi _i}\otimes 1
\\
&-\left( -1 \right) ^{\left| \frac{\partial \left( \xi _I \right)}{\partial \xi _j} \right|-1}\sum_{s\in \llrr\setminus \left\{ j \right\}}^{}{t^m\frac{\partial \left( \xi _I \right)}{\partial \xi _s\partial \xi _j}\frac{\partial}{\partial \xi _i}\otimes E_{s,0}}
\\
&-\left( -1 \right) ^{1+|I|}k\left( k-1 \right) t^{m-2}\xi _I\frac{\partial}{\partial \xi _i}\otimes E_{0,j}+k\sum_{s\in \llrr\setminus \left\{ j \right\}}^{}{t^{m-1}\frac{\partial \left( \xi _I \right)}{\partial \xi _s}\frac{\partial}{\partial \xi _i}\otimes E_{s,j}}
\\
&+\left( m-k \right) t^{m-1}\frac{\partial \left( \xi _I \right)}{\partial \xi _j}\frac{\mathrm{d}}{\mathrm{d}t}\otimes E_{0,i}+\left( m-k-1 \right) \left( m-k \right) t^{m-2}\frac{\partial \left( \xi _I \right)}{\partial \xi _j}\otimes E_{0,i}
\\
&+\left( m-k \right) kt^{m-2}\frac{\partial \left( \xi _I \right)}{\partial \xi _j}\otimes \left( E_{0,0}+E_{j,j} \right) E_{0,i}
\\
&+\left( -1 \right) ^{\left| \frac{\partial \left( \xi _I \right)}{\partial \xi _j} \right|-1}\left( m-k \right) \sum_{s\in \llrr\setminus \left\{ j \right\}}^{}{t^{m-1}\frac{\partial \left( \xi _I \right)}{\partial \xi _s\partial \xi _j}\otimes E_{s,0}E_{0,i}}
\\
&+\left( -1 \right) ^{1+|I|}\left( m-k \right) kt^{m-2}\xi _I\frac{\partial}{\partial \xi _j}\otimes E_{0,i}
\\
&+\left( -1 \right) ^{1+|I|}\left( m-k \right) k\left( k-1 \right) t^{m-3}\xi _I\otimes E_{0,j}E_{0,i}
\\
&+\left( m-k \right) k\sum_{s\in \llrr\setminus \left\{ j \right\}}^{}{t^{m-2}\frac{\partial \left( \xi _I \right)}{\partial \xi _s}\otimes E_{s,j}E_{0,i}}
.
\end{align*}
Then we observe that
\begin{equation}\label{eq4}
  \begin{aligned}
    &-\left( -1 \right) ^{\left| I \right|}\pi \left( t^k\frac{\partial \left( \xi _I \right)}{\partial \xi _j}
\frac{\mathrm{d}}{\mathrm{d}t}+\left( -1 \right) ^{1+|I|}kt^{k-1}\xi _I\frac{\partial}{\partial \xi _j} \right)
\cdot \pi \left( t^{m+3-k}\frac{\partial}{\partial \xi _i} \right)
\\
={}&k^3\left( t^{m}\xi _I\otimes E_{0,j}E_{0,i} \right) +k^2z_2+kz_1+z_0,
  \end{aligned}
 \end{equation}
where $z_0,z_1,z_2\in K_{1,n}\otimes U(\mathfrak{sl}(1,n))$ are independent of $k$.
 Setting $k=0,1,2,3$ in (\ref{eq4}),
we obtain a linear system of equations whose coefficient matrix is nonsingular.
Hence, we have
\[
 t^{m}\xi _I\otimes E_{0,j}E_{0,i} \in \pi(U(\mathcal{S}^\prime(1,n)))
\]
for all $I\subseteq\llrr$, $m\in\mathbb{Z}$ and $i,j\in\llrr$.

Since
$
\pi \left( \frac{\mathrm{d}}{\mathrm{d}t} \right) =\frac{\mathrm{d}}{\mathrm{d}t}\otimes 1\in \pi(U(\mathcal{S}^\prime(1,n)))$
and $
\pi \left( \frac{\partial}{\partial \xi _k} \right) =\frac{\partial}{\partial \xi _k}\otimes 1\in\pi(U(\mathcal{S}^\prime(1,n)))
$
for all $k\in\llrr$,
it follows that $K_{1,n}\otimes E_{0,j}E_{0,i}\subseteq \pi(U(\mathcal{S}^\prime(1,n)))$ for all $i,j\in\llrr$, as desired.
\end{proof}

Using the above lemma, we obtain the following result.

\begin{lemma}\label{l5}
Let $P$ be a simple $K_{1,n}$-module and $V$ be a simple $\mathfrak{sl}(1,n)$-module.
Assume that the $\mathcal{S}^\prime(1,n)$-module $F(P,V)$ is not simple. Then we have $E_{0,j}E_{0,i}V=0$
for all $i,j\in\llrr$.
\end{lemma}
\begin{proof}
  Let \( F' \) be
a nonzero proper submodule of \( F(P, V) \). Let \( \sum_{p=1}^q u_p \otimes v_p \) be a nonzero homogeneous
vector in \( F' \), where \( u_p \)'s are homogeneous vectors in \( P \) and linearly independent,
and \( v_p \)'s are homogeneous vectors in \( V \).

By Lemma \ref{l3}, we have
\[
  x\otimes E_{0,j}E_{0,i}\in \pi(U(\mathcal{S}^\prime(1,n)))
\]
for all $x\in K_{1,n}$ and $i,j\in\llrr$. Hence, we obtain that
\begin{equation}\label{eq6}
   \left( x\otimes E_{0,j}E_{0,i} \right) \cdot \sum_{p=1}^q{u_p\otimes v_p}=\sum_{p=1}^q xu_p \otimes E_{0,j}E_{0,i}v_p\in F^\prime
\end{equation}
for all $x\in K_{1,n}$ and $i,j\in\llrr$.

Since $P$ is a simple $K_{1,n}$-module, by the Jacobson density theorem, for any $u\in P$
and $p\in\llrr[1][q]$,
there is some $x_p\in K_{1,n}$ such that $x_pu_l=\delta_{l,p}u$ for all $l\in\llrr[1][q]$. By (\ref{eq6}), we have
\begin{equation}\label{eq7}
  P\otimes E_{0,j}E_{0,i}v_p\subseteq F^\prime
\end{equation}
for all $i,j\in\llrr$ and $p\in\llrr[1][q]$.

Let $M=\{v\in V|P\otimes v\subseteq F^\prime\}$. For any $u\in P$, $v\in M$ and $l,r,s\in\llrr$ with $r\ne s$, we have
\begin{align*}
  u\otimes \left( E_{0,0}+E_{l,l} \right) v
={}&\pi \left( t\frac{\mathrm{d}}{\mathrm{d}t}+\xi _l\frac{\partial}{\partial \xi _l} \right) \cdot \left( u\otimes v \right) -\left( t\frac{\mathrm{d}}{\mathrm{d}t}+\xi _l\frac{\partial}{\partial \xi _l} \right) u\otimes v
\in F^\prime,
\\
u\otimes E_{r,s}v={}&\pi \left( \xi _r\frac{\partial}{\partial \xi _s} \right) \cdot \left( u\otimes v \right) -\xi _r\frac{\partial}{\partial \xi _s}u\otimes v
\in F^\prime,
\\
u\otimes E_{0,l}v={}&\left( -1 \right) ^{|u|}\left( \pi \left( t\frac{\partial}{\partial \xi _l} \right) \cdot \left( u\otimes v \right) -t\frac{\partial}{\partial \xi _l}u\otimes v \right)
\in F^\prime,\\
u\otimes E_{l,0}v={}&\left( -1 \right) ^{|u|}\left( \xi _l\frac{\mathrm{d}}{\mathrm{d}t}u\otimes v-\pi \left( \xi _l\frac{\mathrm{d}}{\mathrm{d}t} \right) \cdot \left( u\otimes v \right) \right)
\in F^\prime.
\end{align*}
We see that $M$ is an $\mathfrak{sl}(1,n)$-submodule of $V$. The fact that $P\otimes M\subseteq F^\prime\subsetneq F(P,V)$
implies $M\subsetneq V$, which forces that $M=0$. Hence, from (\ref{eq7}), we have $E_{0,j}E_{0,i}v_p=0$ for all $i,j\in\llrr$ and $p\in\llrr[1][q]$.

Let $l\in\llrr$. We have
\[\pi \left( t\frac{\mathrm{d}}{\mathrm{d}t}+\xi _l\frac{\partial}{\partial \xi _l}
\right) \cdot \left( \sum_{p=1}^q{u_p\otimes v_p} \right) =\sum_{p=1}^q{\left( t\frac{\mathrm{d}}{\mathrm{d}t}+\xi _l\frac{\partial}{\partial \xi _l} \right) u_p\otimes v_p}
+\sum_{p=1}^q{u_p\otimes \left( E_{0,0}+E_{l,l} \right) v_p}\in F^{\prime}.
\]
By Lemma \ref{l3}, we obtain that
\begin{align*}
  &\sum_{p=1}^q{x\left( t\frac{\mathrm{d}}{\mathrm{d}t}+\xi _l\frac{\partial}{\partial \xi _l} \right) u_p\otimes E_{0,j}E_{0,i}v_p}+\sum_{p=1}^q{xu_p\otimes E_{0,j}E_{0,i}\left( E_{0,0}+E_{l,l} \right) v_p}
\\
  ={}&\sum_{p=1}^q{xu_p\otimes E_{0,j}E_{0,i}\left( E_{0,0}+E_{l,l} \right) v_p}\in F^{\prime}
\end{align*}
for all $i,j,l\in\llrr$ and $x\in K_{1,n}$. Using the Jacobson density theorem as before, we deduce that
\[P\otimes E_{0,j}E_{0,i}\left( E_{0,0}+E_{l,l} \right) v_p\subseteq F^{\prime}\]
for all $p\in\llrr[1][q]$ and $i,j,l\in\llrr$.
Similarly, from $\pi \left(\xi _r\frac{\partial}{\partial \xi _s}
\right) \cdot \left( \sum_{p=1}^q{u_p\otimes v_p} \right)$, $\pi \left(t\frac{\partial}{\partial \xi _l}
\right) \cdot \left( \sum_{p=1}^q{u_p\otimes v_p} \right)$ and $\pi \left(\xi _l\frac{\mathrm{d}}{\mathrm{d}t}\right) \cdot \left( \sum_{p=1}^q{u_p\otimes v_p} \right) \in F'$, we have
\[
P\otimes E_{0,j}E_{0,i}E_{r,s}v_p\subseteq F^{\prime}\, ,\,
P\otimes E_{0,j}E_{0,i}E_{0,l}v_p\subseteq F^{\prime}\, ,\,
P\otimes E_{0,j}E_{0,i}E_{l,0}v_p\subseteq F^{\prime}
\]
for all $p\in\llrr[1][q]$ and $l,r,s\in\llrr$ with $r\ne s$.
Hence, from $M=0$, we obtain that
\[E_{0,j}E_{0,i}\left( E_{0,0}+E_{l,l} \right) v_p=E_{0,j}E_{0,i}E_{r,s}v_p=E_{0,j}E_{0,i}E_{0,l}v_p=E_{0,j}E_{0,i}E_{l,0}v_p=0\]
for all $p\in\llrr[1][q]$ and $l,r,s\in\llrr$ with $r\ne s$.
By repeating the above procedure, it follows that
\[
E_{0,j}E_{0,i}U(\mathfrak{sl}(1,n))v_p=0\,,\,\forall i,j\in\llrr,p\in\llrr[1][q].
\]
Since $V$ is a simple $\mathfrak{sl}(1,n)$-module,
we have
\[E_{0,j}E_{0,i}U(\mathfrak{sl}(1,n))v_p=E_{0,j}E_{0,i}V=0\,,\,\forall i,j\in\llrr,\]
which completes the proof.
\end{proof}

 Note that, there is a Lie algebra isomorphism $\sigma:\mathfrak{gl}(n)\rightarrow\mathfrak{sl}(1,n)_{0}$
given by
\begin{equation}\label{eq11}
A   \mapsto \left(\begin{smallmatrix}
  \mathrm{tr}\left( A \right)&		0\\
	0&		A\\
 \end{smallmatrix}\right)\,\,,\,\,\forall A\in \mathfrak{gl}(n).
\end{equation} Hence, any $\mathfrak{gl}(n)$-module $V$ can be regarded as an $\mathfrak{sl}(1,n)_{0}$-module by $x v=\sigma (x) v$ for any $x\in \mathfrak{gl}(n)$ and $v\in V$. We also need the following lemma.

\begin{lemma}\label{l6}
 Let $V_0$ be a simple weight $\mathfrak{sl}(1,n)_0$-module. Assume that $E_{0,j}E_{0,i}L(V_0)=0$ for all $i,j\in\llrr$, then
 $V_0\cong N(\lambda)$ for some $\lambda\in \mathbb{C}^n$.
\end{lemma}
\begin{proof}
 Set $\mathrm{Ann}(L(V_0)):=\{x\in U(\mathfrak{sl}(1,n))|xL(V_0)=0\}$. Then $\mathrm{Ann}(L(V_0))$ is an
 ideal of $U(\mathfrak{sl}(1,n))$ and $E_{0,j}E_{0,i}\in\mathrm{Ann}(L(V_0))$ for all $i,j\in\llrr$.

 We have
 \begin{align*}
 \left[ E_{j,0},E_{0,j}E_{0,i} \right]
={}&\left[ E_{j,0},E_{0,j} \right] E_{0,i}-E_{0,j}\left[ E_{j,0},E_{0,i} \right]
\\
={}&\left( E_{j,j}+E_{0,0} \right) E_{0,i}-E_{0,j}E_{j,i}\in \mathrm{Ann}(L(V_0))
 \end{align*}
 for all $i,j\in\llrr$ with $i\ne j$. Then
 \begin{align*}
 &\left[ E_{i,0},\left( E_{j,j}+E_{0,0} \right) E_{0,i}-E_{0,j}E_{j,i} \right] \\
={}&\left[ E_{i,0},\left( E_{j,j}+E_{0,0} \right) E_{0,i} \right] -\left[ E_{i,0},E_{0,j}E_{j,i} \right]
\\
={}&\left[ E_{i,0},\left( E_{j,j}+E_{0,0} \right) \right] E_{0,i}+\left( E_{j,j}+E_{0,0} \right) \left[ E_{i,0},E_{0,i} \right]
\\
&-\left( \left[ E_{i,0},E_{0,j} \right] E_{j,i}-E_{0,j}\left[ E_{i,0},E_{j,i} \right] \right)
\\
={}& E_{i,0}E_{0,i}-E_{0,j}E_{j,0}+\left( E_{j,j}+E_{0,0} \right) \left( E_{i,i}+E_{0,0} \right) -E_{i,j}E_{j,i}
\\
={}& -E_{0,i}E_{i,0}-E_{0,j}E_{j,0}+\left( E_{i,i}+E_{0,0} \right)
\\
&+\left( E_{i,i}+E_{0,0} \right)\left( E_{j,j}+E_{0,0} \right)-E_{i,j}E_{j,i}
\in \mathrm{Ann}(L(V_0))
 \end{align*}
for all $i,j\in\llrr$ with $i\ne j$. Since $E_{i,0} V_0=E_{j,0}V_0=0$ in $L(V_0)$,
it follows that \[\left(\left( E_{i,i}+E_{0,0} \right)
+\left( E_{i,i}+E_{0,0} \right)\left( E_{j,j}+E_{0,0} \right)-E_{i,j}E_{j,i}\right)V_0=0,\]
i.e., $\sigma(\left(E_{i,i}+E_{i,i}E_{j,j}-E_{i,j}E_{j,i}\right))V_0=0$ for all $i,j\in\llrr$ with $i\ne j$. Then from Lemma \ref{l4}, there exists some $\lambda\in \mathbb{C}^n$ such that $V_0\cong N(\lambda)$ as $\mathfrak{sl}(1,n)_0$-modules.
\end{proof}

We now derive the following conclusion.

\begin{theorem}\label{t1}
  Let $P$ be a simple $K_{1,n}$-module and $V$ be a simple weight $\mathfrak{sl}(1,n)$-module.
  Assume that the $\mathcal{S}^\prime(1,n)$-module $F(P,V)$ is not simple, then $V\cong L(N(\lambda))$
  for some $\lambda\in\mathbb{C}^n$ up to a parity-change.
\end{theorem}
\begin{proof}
From Lemma \ref{l1}, $V\cong L(V_0)$ for some simple weight $\mathfrak{sl}(1,n)_0$-module $V_0$ up to a parity-change.
  By Lemma \ref{l5}, we have \[E_{0,j}E_{0,i}L(V_0)=0\] for all $i,j\in\llrr$, and then the conclusion follows from
  Lemma \ref{l6}.
\end{proof}

\subsection{Simple subquotients of \texorpdfstring{$F(P,L(N(\lambda)))$}{F(P, L(N(λ)))}}

Let $P$ be a simple $K_{1,n}$-module and $V$ be a simple weight $\mathfrak{sl}(1,n)$-module.
By Theorem \ref{t1}, the $\mathcal{S}^\prime(1,n)$-module $F(P,V)$ is simple if $V\ncong L(N(\lambda))$ for all $\lambda\in\mathbb{C}^n $ up to a parity-change.
Therefore, in this subsection, we will study the structure of the $\mathcal{S}^\prime(1,n)$-module $F(P,L(N(\lambda)))$
for all $\lambda\in\mathbb{C}^n $.

We first give a concrete realization of the $\mathfrak{sl}(1,n)$-module $L(N(\lambda))$ for $\lambda\in\mathbb{C}^n $.
The Lie superalgebras $\mathfrak{gl}(1,n)$ as well as $\mathfrak{sl}(1,n)$ can be embedded into $K^+_{n,1}$ as  subalgebras
via an injective homomorphism $\iota:\mathfrak{gl}(1,n)\hookrightarrow K^+_{n,1}$ given by
\[
E_{0,0}\rightarrow \xi \frac{\partial}{\partial \xi}
\,\,,\,\,
E_{i,j}\rightarrow t_i\frac{\partial}{\partial t_j}
\,\,,\,\,
E_{0,l}\rightarrow \xi \frac{\partial}{\partial t_l}
\,\,,\,\,
E_{l,0}\rightarrow t_l\frac{\partial}{\partial \xi}
\,\,,\forall i,j,l\in\llrr.
\]
The homomorphism $\iota$ induces a superalgebra homomorphism from $U(\mathfrak{gl}(1,n))$ to $K^+_{n,1}$, which we still denote by $\iota$. Then any $K^+_{n,1}$-module can be regarded as a $\mathfrak{gl}(1,n)$-module via $\iota$.

For a given $\lambda\in \mathbb{C}^n$, consider the following subspace of the simple weight $K^+_{n,1}$-module $A_{n,1}^+(\lambda)$:
\[N_{n,1}\left( \lambda \right) :=\left\{ v\in A_{n,1}^{+}\left( \lambda \right) \,\middle|\,
 \left( \xi \frac{\partial}{\partial \xi}+\sum_{s=1}^n{t_s\frac{\partial}{\partial t_s}} \right) v=|\lambda |v \right\} .\]
According to \cite{CLX,XW},
$N_{n,1}(\lambda)$ is a simple weight module over both $\mathfrak{gl}(1,n)$ and $\mathfrak{sl}(1,n)$.
And $\left( N_{n,1}\left( \lambda \right)\right)_{\bar{0}}=N(\lambda)$.  Note that there exists some $\lambda^\prime\in\Gamma(\lambda)$ with $|\lambda^\prime|=|\lambda|-1$ if $\lambda\ne \mathbf{0}$. We have
\begin{align}\label{eqofN}
 N_{n,1}\left( \lambda \right)=\begin{cases}
  \,\,\C ,&\text{if}\quad \lambda =\mathbf{0},\\
  \,\,N(\lambda)\oplus N(\lambda')\xi ,& \text{if}\quad \lambda\neq \mathbf{0},
\end{cases}
\end{align}
where $\lambda'\in \Gamma(\lambda)$ with $|\lambda'|=|\lambda|-1$. Clearly, $N_{n,1}(\lambda)\cong L(N(\lambda))$.

  Let $P$ be a $K_{1,n}$-module and $P^\prime$ be a $K_{n,1}^+$-module. Then $P'$ is regarded as $\mathfrak{gl}(1,n)$-module via $\iota$ 
  and we have the Shen-Larsson $\mathcal{W}(1,n)$-module $F(P,P')$.
Define
\[\mathrm{diff}:=
\frac{\mathrm{d}}{\mathrm{d}t}\otimes \xi -\sum_{s=1}^n{\frac{\partial}{\partial \xi _s}\otimes t_s}\in K_{1,n}\otimes K_{n,1}^{+}.\]
According to \cite{CLX}, $[\mathrm{diff},\pi(x)]=0$ for all $x\in\mathcal{W}(1,n)$, and hence
$\mathrm{diff}$ is an endomorphism of the $\mathcal{W}(1,n)$-module $F(P,P^\prime)$.

\begin{remark}\label{r2}
  For each $\lambda\in\mathbb{C}^n$, all the modules \[F(P, N_{n,1}(\lambda^\prime))\,\,,\,\,\mathrm{diff}(F(P,N_{n,1}(\lambda^\prime)))
  \,\,,\,\,\forall \lambda^\prime\in\Gamma(\lambda),\]
  are $\mathcal{S}^\prime(1,n)$-submodules of $F(P,A_{n,1}^+(\lambda))$.

Set $\mathbf{-1}:=(-1,-1,\cdots,-1)\in \C^n$.
If $\lambda\ne \mathbf{-1}$,
there exists some $\lambda^\prime\in \Gamma(\lambda)$ with $|\lambda^\prime|=|\lambda|+1$. Then
\[\mathrm{diff}(F(P,N_{n,1}(\lambda)))\subseteq F(P,N_{n,1}(\lambda^\prime)).\]
If $\lambda = \mathbf{-1}$, there does not exist such $\lambda^\prime\in \Gamma(\lambda)$. However,
$$\mathrm{diff}(F(P,N_{n,1}(\mathbf{-1})))\subseteq P\otimes t_1^{-1}t_2^{-1}\cdots t_n^{-1} \xi\subseteq F(P,A^+_{n,1}(\mathbf{-1})).$$
 Note that $\C t_1^{-1}t_2^{-1}\cdots t_n^{-1} \xi$ is a trivial $\mathfrak{sl}(1,n)$-submodule of $A^+_{n,1}(\mathbf{-1})$. Therefore, $\mathrm{diff}(F(P,N_{n,1}(\mathbf{-1})))$ is isomorphic to a $\mathcal{S}^\prime(1,n)$-submodule of $F(P,N_{n,1}(\mathbf{0}))$ up to a parity-change.

Conversely, if $\lambda\ne \mathbf{0}$, there exists some $\lambda^\prime\in\Gamma(\lambda)$ with $|\lambda^\prime|=|\lambda|-1$. Then
\[\mathrm{diff}(F(P,N_{n,1}(\lambda^\prime)))\subseteq F(P,N_{n,1}(\lambda)).\]
\end{remark}

A module $M$ over an algebra $A$ is called trivial if $M\ne 0$ and $A M=0$.

\begin{remark}\label{r3}
For $\lambda\in \C^n$, define \[
\tilde{F}(P,\lambda):=
\left\{ v\in F\left( P,N_{n,1}\left( \lambda \right) \right) \,\middle|\, \mathrm{diff}\left( v \right) =0 \right\} ,
\]
where $P$ is a $K_{1,n}$-module.

Since $\mathrm{diff}\,^2=0$, for any $\lambda\in \C^n$ with $\lambda\ne \mathbf{0}$,
we have \[\mathrm{diff}(F(P,N_{n,1}(\lambda^{\prime})))\subseteq \tilde{F}(P,\lambda)\]
and
the quotient
\[ \tilde{F}(P,\lambda) /\mathrm{diff}(F(P,N_{n,1}(\lambda'))) \]
is either zero or a trivial $\mathcal{S}'(1,n)$-module, where $\lambda^{\prime}\in \Gamma(\lambda)$ with $|\lambda^\prime|=|\lambda|-1$.
\end{remark}

For $i\in\llrr$,
let $e_i\in \C^n$ be the vector whose
$i$-th entry is $1$ and all other entries are $0$.
We have the following result.

\begin{proposition}\label{p3}
  Let $P$ be a simple $K_{1,n}$-module and $\lambda\in \mathbb{C}^n$ with $\lambda\ne \mathbf{0},e_1,e_2,\cdots,e_n$.
  Then for any nonzero $\mathcal{S}^\prime(1,n)$-submodule  $V$ of $F(P,N_{n,1}(\lambda))$, we have
  \[0\ne P_1\otimes \left(N_{n,1}(\lambda)\right)_{\bar{1}}\subseteq V,\]
  where $P_1:=\left\{ p\in P \,\middle|\, \frac{\partial}{\partial \xi _i}p=0 , \forall i\in\llrr \right\} $
  is a simple $K_{1,0}$-module and $\left(N_{n,1}(\lambda)\right)_{\bar{1}}$ is a simple $\mathfrak{sl}(1,n)_{0}$-module.
  Thus, $F(P,N_{n,1}(\lambda))$ has a unique simple submodule.
\end{proposition}
\begin{proof}
  Let \( \sum_{p=1}^q u_p \otimes v_p \) be a nonzero homogeneous
vector in \( V \), where \( \{u_p|p\in\llrr[1][q]\} \) are homogeneous vectors in \( P \) and linearly independent,
and \( \{v_p|p\in\llrr[1][q]\} \) are homogeneous vectors in \( N_{n,1}(\lambda) \).

For $i\in\llrr$, we have
\[\pi \left( \frac{\partial}{\partial \xi _i} \right) \cdot \sum_{p=1}^q{u_p}\otimes v_p=\sum_{p=1}^q{\frac{\partial}{\partial \xi _i}u_p}\otimes v_p
\in V.\]
By Theorem \ref{t2} (4), without loss of generality, we may assume that
\[
\frac{\partial}{\partial \xi_i} u_p = 0\,\,, \,\,\forall \, i \in \llrr\,\,, \,\, p \in \llrr[1][q].
\]
Namely, each $u_p\in P_1$.

For $i\in\llrr$, we have
\begin{align*}
  \pi \left( t\frac{\partial}{\partial \xi _i} \right) \cdot \sum_{p=1}^q{u_p}\otimes v_p={}&\sum_{p=1}^q{t\frac{\partial}{\partial \xi _i}u_p}\otimes v_p+\sum_{p=1}^q{u_p}\otimes E_{0,i}v_p
\\
={}&\sum_{p=1}^q{u_p}\otimes E_{0,i}v_p\in V
.
\end{align*}
Thus, we obtain that
\[\sum_{p=1}^q{u_p}\otimes xv_p\subseteq V, ~\forall x\in U\left( \mathfrak{s} \mathfrak{l} \left( 1,n \right) _1 \right) .\]
 From \eqref{eqofN} and the simplicity of $N_{n,1}(\lambda)$, we
know that, there exists some $x\in U\left( \mathfrak{s} \mathfrak{l} \left( 1,n \right) _1 \right)$
such that $xv_p\in N(\lambda')\xi$ for all $1\leq p\leq q$ and $xv_p$ is nonzero for some $p$.
Thus, replace $v_p$ with $xv_p$ if necessary, we assume that each $0\neq v_p\in N(\lambda')\xi $.

For any \( k \in \mathbb{Z} \) and $l,j\in\llrr$ with $l\ne j$, we have
\begin{align*}
  \pi\left(t^k\xi _l\frac{\partial}{\partial \xi _j}\right)\cdot \left( \sum_{p=1}^q{u_p}\otimes v_p \right) ={}&\sum_{p=1}^q{\left( t^k\xi _l\frac{\partial}{\partial \xi _j}u_p\otimes v_p+kt^{k-1}\xi _lu_p\otimes E_{0,j}v_p+t^ku_p\otimes E_{l,j}v_p \right)}
\\
={}&\sum_{p=1}^q{t^ku_p}\otimes E_{l,j}v_p\in V.
\end{align*}

For any $i\in\llrr$, the element $\kappa_i:=t\frac{\mathrm{d}}{\mathrm{d}t}+ \xi _i\frac{\partial}{\partial \xi _i}\in
\mathcal{S}^\prime(1,n)$ and
we have
\[\pi \left( \kappa _i \right) =\pi \left( t\frac{\mathrm{d}}{\mathrm{d}t}+\xi _i\frac{\partial}{\partial \xi _i}
\right) =t\frac{\mathrm{d}}{\mathrm{d}t}\otimes 1+1\otimes (E_{0,0}+E_{i,i})+ \xi _i\frac{\partial}{\partial \xi _i}
\otimes 1.\]
Set $\kappa:=\kappa_1+\kappa_2+\cdots+\kappa_n$. Then
\[\pi(\kappa)=nt\frac{\mathrm{d}}{\mathrm{d}t}\otimes 1
+ \sum_{i=1}^{n}\xi _i\frac{\partial}{\partial \xi _i}\otimes 1+1\otimes \left(nE_{0,0}+\sum_{i=1}^{n}E_{i,i} \right).\]
Note that $nE_{0,0}+\sum_{i=1}^{n}E_{i,i}$ acts on $N(\lambda')\xi$ by a scalar $b=|\lambda|+n-1$.
Thus, we obtain
\[\pi(\kappa) \cdot \left( \sum_{p=1}^q{t^ku_p\otimes E_{l,j}v_p} \right)- b\sum_{p=1}^q{t^ku_p\otimes E_{l,j}v_p}
=n\sum_{p=1}^q(t\frac{\mathrm{d}}{\mathrm{d}t}) t^{k}u_p\otimes E_{l,j}v_p\in V.\]
It follows that
\[\sum_{p=1}^q{zu_p\otimes E_{l,j}v_p}\in V\]
for all $z\in K_{1,0}=\C[t^{\pm 1},t\frac{\mathrm{d}}{\mathrm{d}t}]$ and $l,j\in\llrr$ with $l\ne j$.
Since $P_1$ is a simple $K_{1,0}$-module, by
the Jacobson density theorem, for any $u\in P_1$, there exists an element $y\in K_{1,0}$ such that
$yu_p=\delta_{1,p}u$ for all $p\in\llrr$. Therefore,
\begin{equation}\label{eqofP_1}
P_1\otimes E_{l,j}v_1\subseteq V,  ~ ~ \forall l,j\in\llrr~ \textrm{with}~ l\ne j.
\end{equation}

Let $M=\{w\in N(\lambda')\xi~|~ P_1\otimes w\subseteq V\}$. Since $\lambda\ne e_1,e_2,\cdots,e_n$, it follows that $\lambda'\neq \mathbf{0}$ and
there exist some $l_0,j_0\in\llrr$ with $l_0\ne j_0$ such that $0\neq E_{l_0,j_0}v_1=t_{l_0}\frac{\partial}{\partial t_{j_0}}v_1\in N(\lambda')\xi$.
Then from \eqref{eqofP_1}, we have $E_{l_0,j_0}v_1\in M$ and $M\neq 0$. For any $r,s\in\llrr$ with $r\ne s$, $u\in P_1$ and $w\in M$, we have
\[ \pi\left(\xi _r\frac{\partial}{\partial \xi _s}\right)\cdot \left( u\otimes w \right) =
u\otimes E_{r,s}w\in V\]
and
\[  \pi\left( t\frac{\mathrm{d}}{\mathrm{d}t}+\xi _r\frac{\partial}{\partial \xi _r} \right) 
  \cdot \left( u\otimes w \right) =t\frac{\mathrm{d}}{\mathrm{d}t}u\otimes w+u\otimes \left( E_{0,0}+E_{r,r} \right) w\in V\]
which implies $E_{r,s}w, \left( E_{0,0}+E_{r,r} \right) w\in M$. So $M$ is a nonzero submodule of the simple $\mathfrak{sl}(1,n)_{0}$-module $N(\lambda')\xi$. Thus, $M=N(\lambda')\xi$ and $P_1\otimes (N_{n,1}(\lambda))_{\bar{1}}=P_1\otimes N(\lambda')\xi\subseteq V$. Then the intersection of all nonzero submodules of $F(P, N_{n,1}(\lambda))$ is nonzero and hence it is the unique simple submodule of $F(P, N_{n,1}(\lambda))$.
\end{proof}

\begin{lemma}\label{l7}
  For all $m\in\mathbb{Z}$, $I\subseteq \llrr$ and $i,j\in\llrr$, we have
  $a(m,I,i,j)\in \pi(\mathcal{S}^\prime(1,n))$, where
  \begin{align*}
 a(m,I,i,j) := {}& \left( -1 \right) ^{|I|}t^{m}\xi _I\frac{\partial}{\partial \xi _i}\otimes E_{0,j}+t^{m}\frac{\partial}{\partial \xi _j}\xi _I\otimes \left( E_{0,i}-E_{0,0}E_{0,i} \right)+\left( -1 \right) ^{\left| I \right|}t^{m}\xi _I\frac{\partial}{\partial \xi _j}\otimes E_{0,0}E_{0,i}
\\
&-\sum_{s\in \llrr}^{}{\frac{\partial}{\partial \xi _s} t^{m}\xi _I\otimes E_{s,j}E_{0,i}}+\left( -1 \right) ^{\left| I \right|}\sum_{s\in \llrr}^{}{t^{m}\xi _I\frac{\partial}{\partial \xi _s}\otimes E_{s,j}E_{0,i}}
.
\end{align*}
\end{lemma}

\begin{proof}
  Let $m,k\in\mathbb{Z}$, $I\subseteq\llrr$ and $i,j\in\llrr$.
  Recall the calculations in the proof of Lemma \ref{l3}.
We observe that
  \begin{equation}\label{eq8}
    \begin{aligned}
    &\pi \left( t^k\frac{\partial \left( \xi _I \right)}{\partial \xi _j}\frac{\mathrm{d}}{\mathrm{d}t}+\left( -1 \right) ^{1+|I|}kt^{k-1}\xi _I\frac{\partial}{\partial \xi _j} \right) \cdot \pi \left( t^{m+2-k}\frac{\partial}{\partial \xi _i} \right)
\\
={}&k^3z_3+k^2a^{\prime}\left( m,I,i,j \right) +kz_1+k_0,
  \end{aligned}
  \end{equation}
  where $z_3,a^{\prime}\left( m,I,i,j \right),z_1,z_0\in K_{1,n}\otimes U(\mathfrak{sl}(1,n))$ are independent of $k$ and
  \begin{align*}
  a^{\prime}\left( m,I,i,j \right)={}&\left( -1 \right) ^{|I|}t^{m}\xi _I\frac{\partial}{\partial \xi _i}\otimes E_{0,j}
+t^{m}\frac{\partial \left( \xi _I \right)}{\partial \xi _j}\otimes E_{0,i}
\\
&-t^{m}\frac{\partial \left( \xi _I \right)}{\partial \xi _j}\otimes \left( E_{0,0}+E_{j,j} \right) E_{0,i}+\left( -1 \right) ^{|I|}t^{m}\xi _I\frac{\partial}{\partial \xi _j}\otimes E_{0,i}
\\
&-\left( -1 \right) ^{|I|}\left( 3+m \right) t^{m-1}\xi _I\otimes E_{0,j}E_{0,i}
-\sum_{s\in \llrr\setminus \left\{ j \right\}}^{}{t^{m}\frac{\partial \left( \xi _I \right)}{\partial \xi _s}\otimes E_{s,j}E_{0,i}}
.
  \end{align*}
 Setting $k=0,1,2,3$ in (\ref{eq8}),
we obtain a linear system of equations whose coefficient matrix is nonsingular.
Hence, we have
\[
 a^{\prime}\left( m,I,i,j \right) \in \pi(U(\mathcal{S}^\prime(1,n)))
\]
for all $I\subseteq\llrr$, $m\in\mathbb{Z}$ and $i,j\in\llrr$.

Furthermore, by Lemma \ref{l3}, we know that
\[
 a^{\prime}\left( m,I,i,j \right)+\left( -1 \right) ^{|I|}\left( 3+m \right) t^{m-1}\xi _I\otimes E_{0,j}E_{0,i}
  \in \pi(U(\mathcal{S}^\prime(1,n)))
\]
for all $I\subseteq\llrr$, $m\in\mathbb{Z}$ and $i,j\in\llrr$. Then we have
\begin{align*}
  & a^{\prime}\left( m,I,i,j \right) +\left( -1 \right) ^{|I|}\left( 3+m \right) t^{m-1}\xi _I\otimes E_{0,j}E_{0,i}
\\
={}&\left( -1 \right) ^{|I|}t^m\xi _I\frac{\partial}{\partial \xi _i}\otimes E_{0,j}
+t^m\frac{\partial \left( \xi _I \right)}{\partial \xi _j}\otimes E_{0,i}
\\
&-t^m\frac{\partial \left( \xi _I \right)}{\partial \xi _j}\otimes \left( E_{0,0}+E_{j,j} \right) E_{0,i}
\\
&+\left( -1 \right) ^{|I|}t^m\xi _I\frac{\partial}{\partial \xi _j}\otimes E_{0,i}
\\
&-\sum_{s\in \llrr\setminus \left\{ j \right\}}^{}{t^m\frac{\partial \left( \xi _I \right)}{\partial \xi _s}\otimes E_{s,j}E_{0,i}}
\\
={}&\left( -1 \right) ^{|I|}t^m\xi _I\frac{\partial}{\partial \xi _i}\otimes E_{0,j}+\left( -1 \right) ^{|I|}t^m\xi _I\frac{\partial}{\partial \xi _j}\otimes E_{0,i}
\\
&+t^m\left( \frac{\partial}{\partial \xi _j}\xi _I-\left( -1 \right) ^{|I|}\xi _I\frac{\partial}{\partial \xi _j} \right) \otimes E_{0,i}
\\
&-t^m\left( \frac{\partial}{\partial \xi _j}\xi _I-\left( -1 \right) ^{|I|}\xi _I\frac{\partial}{\partial \xi _j} \right) \otimes \left( E_{0,0}+E_{j,j} \right) E_{0,i}
\\
&-\sum_{s\in \llrr\setminus \left\{ j \right\}}^{}{t^m\left( \frac{\partial}{\partial \xi _s}\xi _I-\left( -1 \right) ^{|I|}\xi _I\frac{\partial}{\partial \xi _s} \right) \otimes E_{s,j}E_{0,i}}
\\
={}&\left( -1 \right) ^{|I|}t^m\xi _I\frac{\partial}{\partial \xi _j}\otimes E_{0,0}E_{0,i}
\\
&+\frac{\partial}{\partial \xi _j}t^m\xi _I\otimes \left( E_{0,i}-E_{0,0}E_{0,i} \right) +\left( -1 \right) ^{|I|}t^m\xi _I\frac{\partial}{\partial \xi _i}\otimes E_{0,j}
\\
&-\sum_{s\in \llrr}^{}{\frac{\partial}{\partial \xi _s}t^m\xi _I\otimes E_{s,j}E_{0,i}}+\left( -1 \right) ^{|I|}\sum_{s\in \llrr}^{}{t^m\xi _I\frac{\partial}{\partial \xi _s}\otimes E_{s,j}E_{0,i}}
\\
={}& a\left( m,I,i,j \right)\in\pi(\mathcal{S}^\prime(1,n))
\end{align*}
for all $I\subseteq\llrr$, $m\in\mathbb{Z}$ and $i,j\in\llrr$.
\end{proof}

Using the above lemma, we prove the following result.

\begin{proposition}\label{p4}
  Let $P$ be a simple $K_{1,n}$-module and $\lambda\in\mathbb{C}^n$ with $\lambda\ne \mathbf{0}, \mathbf{-1}$.
  Then
  $\mathrm{diff}(F(P,N_{n,1}(\lambda)))$ is a simple $\mathcal{S}^\prime(1,n)$-submodule of $F(P,N_{n,1}(\lambda^\prime))$ for some
  $\lambda^\prime\in\Gamma(\lambda)$ with $|\lambda^\prime|=|\lambda|+1$.
\end{proposition}

\begin{proof}
  Assume that $V\subseteq \mathrm{diff}(F(P,N_{n,1}(\lambda)))$ is a nonzero
  $\mathcal{S}^\prime(1,n)$-submodule of $\mathrm{diff}(F(P,N_{n,1}(\lambda)))$.
  We need to show that $V=\mathrm{diff}(F(P,N_{n,1}(\lambda)))$.

  Since $\lambda\ne \mathbf{-1} $, there exists some
  $\lambda^\prime\in\Gamma(\lambda)$ with $|\lambda^\prime|=|\lambda|+1$ such that
  \[\mathrm{diff}(F(P,N_{n,1}(\lambda)))\subseteq F(P,N_{n,1}(\lambda^\prime)).\]
By the condition, we see $\lambda^\prime\ne \mathbf{0},e_1,e_2,\cdots,e_n$. Thus, we know from Proposition \ref{p3} that
\[0\ne P_1\otimes (N_{n,1}(\lambda^\prime))_{\bar{1}}\subseteq V,\]
where
$P_1:=\left\{ p\in P \,\middle|\, \frac{\partial}{\partial \xi _i}u=0\,\,,\,\,
\forall i\in\llrr
 \right\}$.
Hence, there exists $0\neq w:=\sum_{k=1}^{\ell} p_k\otimes v_k\in V$ such that
 $\left\{p_k\right\}_{k=1}^\ell$ are nonzero homogeneous vectors in \( P_1\setminus\{0\} \) and linearly independent and
 \( \left\{v_k\right\}_{k=1}^\ell \) are nonzero homogeneous vectors in $(N_{n,1}(\lambda^\prime))_{\bar{1}}$.

For $r\in\llrr$, we have
\[\pi \left( \xi _r\frac{\mathrm{d}}{\mathrm{d}t} \right) \cdot
\sum_{k\in \llrr[1][\ell]}^{}{p_k\otimes v_k}=\sum_{k\in \llrr[1][\ell]}^{}{\xi _r\frac{\mathrm{d}}{\mathrm{d}t}p_k\otimes v_k}
+\sum_{k\in \llrr[1][\ell]}{\left( -1 \right) ^{|p_k|}p_k\otimes t_r\frac{\partial}{\partial \xi}v_k}\in V
.
\]
Then,
by Lemma \ref{l7}, we have
  \begin{align*}
 &a\left( m,I,i,j \right) \cdot \left( \sum_{k\in \llrr[1][\ell]}^{}{\xi _r\frac{\mathrm{d}}{\mathrm{d}t}p_k\otimes v_k}
 +\sum_{k\in \llrr[1][\ell]}{\left( -1 \right) ^{|p_k|}p_k\otimes t_r\frac{\partial}{\partial \xi}v_k} \right)
\\
={}& a\left( m,I,i,j \right) \cdot \sum_{k\in \llrr[1][\ell]}{\left( -1 \right) ^{|p_k|}p_k\otimes t_r\frac{\partial}{\partial \xi}v_k}
\\
={}&-\sum_{s\in \llrr}^{}{\sum_{k\in \llrr[1][\ell]}{\frac{\partial}{\partial \xi _s}t^m\xi _Ip_k\otimes t_s\frac{\partial}{\partial t_j}\frac{\partial}{\partial t_i}t_r\xi \frac{\partial}{\partial \xi}v_k}}
\\
={}&-\sum_{s\in \llrr}^{}{\sum_{k\in \llrr[1][\ell]}{\frac{\partial}{\partial \xi _s}t^m\xi _Ip_k\otimes t_s\frac{\partial}{\partial t_j}\frac{\partial}{\partial t_i}t_rv_k}}
\\
={}&-\left( \sum_{s\in \llrr}^{}{\frac{\partial}{\partial \xi _s}\otimes t_s} \right) \cdot \sum_{k\in \llrr[1][\ell]}{t^m\xi _Ip_k\otimes \frac{\partial}{\partial t_j}\frac{\partial}{\partial t_i}t_rv_k}
\in  V
\end{align*}
for all $m\in\mathbb{Z}$, $I\subseteq \llrr$ and $i,j,r\in\llrr$.
Therefore, we obtain
\begin{align*}
 & \mathrm{diff}\left( \sum_{k\in \llrr[1][\ell]}{t^m\xi _Ip_k\otimes \frac{\partial}{\partial t_j}\frac{\partial}{\partial t_i}t_rv_k} \right)
\\
={}&\left( \frac{\mathrm{d}}{\mathrm{d}t}\otimes \xi \right) \cdot \sum_{k\in \llrr[1][\ell]}{t^m\xi _Ip_k\otimes
\frac{\partial}{\partial t_j}\frac{\partial}{\partial t_i}t_rv_k}
-\left( \sum_{s\in \llrr}^{}{\frac{\partial}{\partial \xi _s}\otimes t_s}
\right) \cdot\sum_{k\in \llrr[1][\ell]}{t^m\xi _Ip_k\otimes \frac{\partial}{\partial t_j}\frac{\partial}{\partial t_i}t_rv_k}
\\
={}&-\left( \sum_{s\in\llrr}^{}{\frac{\partial}{\partial \xi _s}\otimes t_s} \right)\cdot \sum_{k\in \llrr[1][\ell]}{t^m\xi _Ip_k\otimes \frac{\partial}{\partial t_j}\frac{\partial}{\partial t_i}t_rv_k}\in V
\end{align*}
for all $m\in\mathbb{Z}$, $I\subseteq \llrr$ and $i,j,r\in\llrr$.
Furthermore, since $[\mathrm{diff},\pi(x)]=0$ for all $ x\in\mathcal{S}^\prime(1,n)$ and
\[\pi \left( \frac{\mathrm{d}}{\mathrm{d}t} \right) =\frac{\mathrm{d}}{\mathrm{d}t}\otimes 1\in \pi(U(\mathcal{S}^\prime(1,n)))
\,\,,\,\,
\pi \left( \frac{\partial}{\partial \xi _s} \right) =\frac{\partial}{\partial \xi _s}\otimes 1\in\pi(U(\mathcal{S}^\prime(1,n)))
\,,\,\forall s\in\llrr,
\]
we derive that
\begin{equation}\label{eq9}
  \mathrm{diff}\left( \sum_{k\in \llrr[1][\ell]}^{}{x p_k\otimes \frac{\partial}{\partial t_j}\frac{\partial}{\partial t_i}t_rv_k} \right)
\in V
\end{equation}
for all $x\in K_{1,n}$ and $i,j,r\in\llrr$.

Since $\lambda^\prime \ne \mathbf{0},e_1,e_2,\cdots,e_n$ and each $v_k\in N^\prime(\lambda^\prime)$, there exists some $k_0\in\llrr[1][\ell]$
and $i_0,j_0,r_0\in \llrr$ such that
\[\zeta:
=\frac{\partial}{\partial t_{j_0}}\frac{\partial}{\partial t_{i_0}}t_{r_0}v_{k_0}\in N_{n,1}(\lambda)\setminus \{0\}.\]
By the Jacobson density
theorem, for any $p\in P$, there exists some $x_p\in K_{1,n}$ such that
\[x_p p_{k_0}=p\quad\mathrm{and}\quad x_p p_{k}=0\,\,,\,\,\forall k\in\llrr[1][\ell]\setminus \{k_0\}.\]
We deduce from (\ref{eq9}) that
\[
   \mathrm{diff}\left(P \otimes \zeta \right)
\subseteq V.
\]

It follows that
\[M:=\left\{ v\in N_{n,1}\left( \lambda \right) \,\middle|\, \mathrm{diff}(P\otimes v)\subseteq V \right\} \]
is a nonzero subspace of $N_{n,1}(\lambda)$. Furthermore, it is easy to see that
$M$ is an $\mathfrak{sl}(1,n)$-submodule of $N_{n,1}(\lambda)$, and thus we deduce from the simplicity of $N_{n,1}(\lambda)$
that
$M=N_{n,1}(\lambda)$.
Therefore, we have $\mathrm{diff}(F(P,N_{n,1}(\lambda)))=V$, as desired.
\end{proof}

From Proposition \ref{p3}, Proposition \ref{p4} and Remark \ref{r2}, we obtain the following result.

\begin{proposition}\label{p5}
  Assume that $P$ is a simple $K_{1,n}$-module and $\lambda\in\mathbb{C}^n$ with $\lambda\ne \mathbf{0}, e_1,e_2,\cdots,e_n$.
  Then the $\mathcal{S}^\prime(1,n)$-module $F(P,N_{n,1}(\lambda))$ has a unique simple submodule
  $\mathrm{diff}(F(P,N_{n,1}(\lambda^\prime)))$, where $\lambda^\prime\in \Gamma(\lambda)$ with $|\lambda^\prime|=|\lambda|-1$.
\end{proposition}

Furthermore, we have the following result.

\begin{theorem}\label{t3}
Assume that $P$ is simple $K_{1,n}$-module and $\lambda\in\mathbb{C}^n$ with
  $\lambda\ne \mathbf{0}, \mathbf{-1}, e_1,e_2,\cdots,e_n$.
  Let $\lambda^\prime\in\Gamma(\lambda)$ with $|\lambda^\prime|=|\lambda|-1$.
  The $\mathcal{S}^\prime(1,n)$-submodule filtration of $F(P,N_{n,1}(\lambda))$
  \[0\subseteq\mathrm{diff}(F(P,N_{n,1}(\lambda^\prime)))\subseteq
  \tilde{F}(P,\lambda)\subseteq
  F(P,N_{n,1}(\lambda))
  \]
 satisfies the following properties.

  (1) $\mathrm{diff}(F(P,N_{n,1}(\lambda^\prime)))$ is the unique simple submodule of $F(P,N_{n,1}(\lambda))$.

  (2) $\tilde{F}(P,\lambda)/\mathrm{diff}(F(P,N_{n,1}(\lambda^\prime)))$ is either zero or a trivial module.

  (3) $F(P,N_{n,1}(\lambda))/\tilde{F}(P,\lambda)\cong \mathrm{diff}(F(P,N_{n,1}(\lambda)))$
  is simple.

Therefore, $F(P,N_{n,1}(\lambda))$ is simple if and only if $F(P,N_{n,1}(\lambda))=\mathrm{diff}(F(P,N_{n,1}(\lambda^\prime)))$.
\end{theorem}

\begin{proof}
  This follows from Proposition \ref{p4}, Proposition \ref{p5}, Remark \ref{r2} and Remark \ref{r3}.
\end{proof}

We now turn to study the other
Shen-Larsson modules of $\mathcal{S}^\prime(1,n)$, which are not contained in Theorem \ref{t3}.
We will begin with
$F(P,N_{n,1}(\mathbf{0}))$, where $P$ is a simple $K_{1,n}$-module.
Note that $N_{n,1}(\mathbf{0})$ is the one-dimensional trivial $\mathfrak{sl}(1,n)$-module, and thus we identify
$F(P,N_{n,1}(\mathbf{0}))=P$. Define
\[\Delta P:=\frac{\mathrm{d}}{\mathrm{d}t}P+\sum_{i=1}^n{\frac{\partial}{\partial \xi _i}P}\subseteq P.\]
For any $p\in P$, $k\in \mathbb{Z}$, $I\subseteq \llrr$ and $i,j\in \llrr$ with $i\neq j$, we have
\begin{align*}
\left( t^k\frac{\partial \left( \xi _I \right)}{\partial \xi _i}\frac{\mathrm{d}}{\mathrm{d}t}-(-1)^{|I|}kt^{k-1}
\xi _I\frac{\partial}{\partial \xi _i} \right) p={}&
\frac{\mathrm{d}}{\mathrm{d}t}t^k\frac{\partial (\xi _I)}{\partial \xi _i}p-k\frac{\partial}{\partial \xi _i}t^{k-1}\xi _Ip,
\\
\left( t^k\frac{\partial \left( \xi _I \right)}{\partial \xi _i}\frac{\partial}{\partial \xi _j}
+t^k\frac{\partial \left( \xi _I \right)}{\partial \xi _j}\frac{\partial}{\partial \xi _i} \right) p
={}&-(-1)^{\left| \frac{\partial (\xi _I)}{\partial \xi _i} \right|}
\frac{\partial}{\partial \xi _j}t^k\frac{\partial \left( \xi _I \right)}{\partial \xi _i}
p-(-1)^{\left| \frac{\partial (\xi _I)}{\partial \xi _j} \right|}\frac{\partial}{\partial \xi _i}
t^k\frac{\partial \left( \xi _I \right)}{\partial \xi _j}p.
\end{align*}
So $\mathcal{S}^\prime(1,n) P\subseteq \Delta P$. Hence, $\Delta P$ is an $\mathcal{S}^\prime(1,n)$-submodule of $P$ and the quotient $P/\Delta P$ is either trivial or zero. For convenience, we denote by $A_{1,n}:=\mathbb{C}[t^{\pm 1}]\otimes \Lambda(n)$ the simple weight $K_{1,n}$-module.

\begin{proposition}\label{p6}
Let $P$ be a simple module over $K_{1,n}$. The following statements hold.

(1)If $P\ncong A_{1,n}, \Pi(A_{1,n})$, then $P$ has a unique simple $\mathcal{S}^\prime(1,n)$-submodule
$\Delta P$ and the quotient $P/\Delta P$ is either trivial or zero. Therefore, $P$ is simple if and only if
$P=\Delta P$.

(2) If $P\cong A_{1,n}$, then $P$ is not a simple $\mathcal{S}^\prime(1,n)$-module. Moreover,
$\C$ and $\Delta A_{1,n}$
are all nonzero proper $\mathcal{S}^\prime(1,n)$-submodules of $A_{1,n}$,
and $\Delta A_{1,n}/\C$ is the unique nontrivial simple subquotient of $A_{1,n}$.

\end{proposition}
\begin{proof}
By Theorem \ref{t2} (4), without loss of generality, we may assume that $P=P_1\otimes \Lambda(n)$, where $P_1$ is a simple $K_{1,0}$-module.

(1) If $P\ncong A_{1,n}$, i.e., $P_1\ncong A_{1,0}$. 
Let $N$ be any nonzero $\mathcal{S}^\prime(1,n)$-submodule of $P$ and $0\neq v=\sum_{I} v_I\otimes\xi_I\in N$.  Let $I_0=\{i_1,\cdots, i_k\}$ be a maximal element in $\{I~|~ v_I\neq 0\}$ with respect to ``$\subseteq$''. By applying $\frac{\partial}{\partial \xi_{i_1}}\cdots \frac{\partial}{\partial \xi_{i_k}}$ to $v$, we deduce that $v_{I_0}\otimes 1\in N$. Then we have
$$(t^k\frac{\partial(\xi_I)}{\partial \xi_i}\frac{\mathrm{d}}{\mathrm{dt}}-(-1)^{|I|}kt^{k-1}\xi_I\frac{\partial}{\partial \xi_i})( v_{I_0}\otimes 1)=t^k\frac{\mathrm{d}}{\mathrm{dt}} v_{I_0}\otimes\frac{\partial(\xi_I)}{\partial \xi_i}\in N,~~\forall i\in \llrr,  I\subseteq \llrr.$$
So $(\frac{\mathrm{d}}{\mathrm{dt}})^k \left(t^k\frac{\mathrm{d}}{\mathrm{dt}} v_{I_0}\right)\otimes\frac{\partial(\xi_I)}{\partial \xi_i}\in N,$
which implies that
\begin{equation*}\label{eqofK}
K_{1,0}\frac{\mathrm{d}}{\mathrm{dt}} v_{I_0}\otimes\frac{\partial(\xi_I)}{\partial \xi_i}\subseteq N, ~~\forall i\in \llrr, I\subseteq \llrr.
\end{equation*}
Since $\frac{\mathrm{d}}{\mathrm{dt}} v_{I_0}\neq 0$ and $P_1$ is simple, it follows that $P_1\otimes \frac{\partial(\xi_I)}{\partial \xi_i}\subseteq  N$, i.e., $\frac{\partial}{\partial \xi_i}(P) \subseteq N$. Moreover, for any $p\in P_1$, $\varnothing \neq I\subseteq \llrr$ and $i\in I$, we have $\xi_{I\backslash \{i\}}\frac{\mathrm{d}}{\mathrm{dt}} (p\otimes\xi_i)=(\frac{\mathrm{d}}{\mathrm{dt}} p)\otimes\xi_I\in N$.
Combining with $\frac{\mathrm{d}}{\mathrm{dt}} P_1\otimes 1\subseteq N$, we have $\frac{\mathrm{d}}{\mathrm{dt}}(P) \subseteq N$.  Hence, $\Delta P\subseteq N$ and $\Delta P$ is the unique simple $\mathcal{S}^\prime(1,n)$-submodule of $P$.

(2) We only need to consider the case that $P=A_{1,n}$.  Let $N$ be any nonzero submodule of $A_{1,n}$ distinct from $\C$ and $v=\sum v_I\xi_I\in A_{1,n}\backslash \C$ with $v_I\in A_{1,0}$.  Let $I_0=\{i_1,\cdots, i_k\}$ be a maximal element in $\{I~|~ v_I\neq 0\}$ with respect to ``$\subseteq$''.  Then $v_{I_0}\notin \C$ if $I_0=\varnothing$. If $I_0\neq \varnothing$, by replacing $v$ with $t\frac{\partial}{\partial \xi_i} v$ and $I_0$ with $I_0\backslash \{i\}$ for $i\in I_0$ if necessary, we also assume that $v_{I_0}\notin \C$. So $\frac{\mathrm{d}}{\mathrm{dt}} v_{I_0}\neq 0$. Then by the same arguments as in $(1)$, we see that $\Delta A_{1,n}\subseteq N$. It is clear that $\C\subseteq \Delta A_{1,n}$. Thus, $\C$ is the unique simple submodule of $A_{1,n}$ and $\Delta A_{1,n}/ \C$ is the unique simple submodule of $A_{1,n}/\C$.
\end{proof}

The following result studies the $\mathcal{S}^\prime(1,n)$-module $F(P,N_{n,1}(e_i))$ for $i\in\llrr$.

\begin{proposition}\label{p99}
  Let $P$ be a simple $K_{1,n}$-module. For $i\in\llrr$,
  the $\mathcal{S}^\prime(1,n)$-submodule filtration of $F(P,N_{n,1}(e_i))$
\[  0\subseteq\mathrm{diff}(F(P,N_{n,1}(\mathbf{0})))\subseteq
  \tilde{F}(P,e_i)\subseteq
  F(P,N_{n,1}(e_i))
\]
satisfies the following properties.

  (1) If $P\cong A_{1,n}$,
  then $\mathrm{diff}(F(P,N_{n,1}(\mathbf{0})))\cong A_{1,n}/\C $, which has a unique nontrivial simple subquotient $\Delta A_{1,n}/\C$.

  (2) If $P\ncong A_{1,n}$,
  then $\mathrm{diff}(F(P,N_{n,1}(\mathbf{0})))\cong F(P,N_{n,1}(\mathbf{0}))=P $, which has a unique simple submodule
$\Delta P$ with the corresponding quotient $P/\Delta P$ being either zero or a trivial module.

(3) $\tilde{F}(P,e_i)/\mathrm{diff}(F(P,N_{n,1}(\mathbf{0})))$ is either zero or a trivial module.

(4) $F(P,N_{n,1}(e_i))/\tilde{F}(P,e_i)\cong \mathrm{diff}(F(P,N_{n,1}(e_i)))$ is simple.

In particular, $F(P,N_{n,1}(e_i))$ is not simple.
\end{proposition}
\begin{proof}
  Note that
  \[\tilde{F}(P,\mathbf{0})=
\mathrm{Ker}(\mathrm{diff}\colon F(P,N_{n,1}(\mathbf{0}))\rightarrow F(P,N_{n,1}(e_i)))
=\left\{ p\in P \,\middle|\, \frac{\mathrm{d}}{\mathrm{d}t}p=\frac{\partial}{\partial \xi _j}p=0 , \forall j\in\llrr \right\}.
\]
Thus, from $K_{1,n}/ (\sum_{i\in \llrr} K_{1,n}\frac{\partial}{\partial \xi_i}+K_{1,n}\frac{\mathrm{d}}{\mathrm{d}t})\cong A_{1,n}$ as $K_{1,n}$-modules, we have that $\tilde{F}(P,\mathbf{0})=0$ if and only if $P\ncong A_{1,n}$.
This implies that $\mathrm{diff}(F(P,N_{n,1}(\mathbf{0})))\cong F(P,N_{n,1}(\mathbf{0}))$
if and only if $P\ncong A_{1,n}$, and then the statement (2) follows from
Proposition \ref{p6} (1).

If $P\cong A_{1,n}$, then
$\mathrm{diff}(F(P,N_{n,1}(\mathbf{0})))\cong A_{1,n}/\C$ and $A_{1,n}/\C$  has a unique nontrivial simple subquotient $\Delta A_{1,n}/\C$
by Proposition \ref{p6} (2). This gives statement (1).

The statements (3) and (4) follow from Proposition \ref{p4} and Remark \ref{r3}.
\end{proof}

The following result studies the $\mathcal{S}^\prime(1,n)$-module $F(P,N_{n,1}(\mathbf{-1}))$.

\begin{proposition}\label{p98}
Let $P$ be a simple $K_{1,n}$-module and $\lambda^\prime\in\Gamma(\mathbf{-1})$ with $|\lambda^\prime|=-1-n$.
The $\mathcal{S}^\prime(1,n)$-submodule filtration of $F(P,N_{n,1}(\mathbf{-1}))$
\[
0\subseteq\mathrm{diff}(F(P,N_{n,1}(\lambda^\prime)))\subseteq
\tilde{F}(P,\mathbf{-1})\subseteq
F(P,N_{n,1}(\mathbf{-1}))
\]
satisfies the following properties.

(1) $\mathrm{diff}(F(P,N_{n,1}(\lambda^\prime)))$ is the unique simple submodule of $F(P,N_{n,1}(\mathbf{-1}))$.

(2) $\tilde{F}(P,\mathbf{-1})/\mathrm{diff}(F(P,N_{n,1}(\lambda^\prime)))$ is either zero or a trivial module.

(3) If $P\ncong A_{1,n}$, then $F(P,N_{n,1}(\mathbf{-1}))/\tilde{F}(P,\mathbf{-1})\cong \Pi (\Delta P)$ is simple.

(4) If $P\cong A_{1,n}$, then $F(P,N_{n,1}(\mathbf{-1}))/\tilde{F}(P,\mathbf{-1})\cong \Pi(\Delta A_{1,n})$,
which has a trivial submodule $\mathbb{C}$ with the corresponding quotient $\Delta A_{1,n}/\mathbb{C}$ being simple.

In particular, $F(P,N_{n,1}(\mathbf{-1}))$ is not simple.

\end{proposition}
\begin{proof}
  Note that
  \begin{align*}
  \mathrm{diff}(F(P,N_{n,1}(\mathbf{-1})))={}&\frac{\mathrm{d}}{\mathrm{d}t} P\otimes t_1^{-1}t_2^{-1}\cdots t_n^{-1} \xi+\sum_{i=1}^n\frac{\partial}{\partial \xi_i} P\otimes t_1^{-1}t_2^{-1}\cdots t_n^{-1} \xi\\
  ={}&\Delta P\otimes t_1^{-1}t_2^{-1}\cdots t_n^{-1} \xi \subseteq P\otimes\C t_1^{-1}t_2^{-1}\cdots t_n^{-1} \xi
   \end{align*}
with $\C t_1^{-1}t_2^{-1}\cdots t_n^{-1} \xi$ is a trivial $\mathfrak{sl}(1,n)$-module. Hence,
  $\mathrm{diff}(F(P,N_{n,1}(\mathbf{-1})))$ is isomorphic to $\Pi(\Delta P)$, which is the
  $\mathcal{S}^\prime(1,n)$-submodule of $\Pi(F(P,N_{n,1}(\mathbf{0})))$.
  Then the statements (3) and (4) follow from Proposition \ref{p6}.
  The statements (1) and (2) are implied by Proposition \ref{p5} and Remark \ref{r3}.
\end{proof}

\section{Cuspidal modules}
In this section, we refine the classification of simple
cuspidal modules over $\mathcal{S}^\prime (1,n)$ obtained in \cite{MMZ}.
We first recall some related notations and results from the reference \cite{MMZ}.

For $i\in\llrr[1][n-1]$, set $h_i=\xi _i\frac{\partial}{\partial \xi _i}
-\xi _{i+1}\frac{\partial}{\partial \xi _{i+1}}$.
Let $H_0:=\oplus_{i=1}^{n-1}\C h_i$ be an abelian subalgebra of $\mathcal{S}^\prime(1,n)$.
An element $\lambda\in H_0^*$ is dominant if and only if
$\lambda(h_i)\in\Z_+$ for all $i\in\llrr[1][n-1]$.

Denote by $\mathrm{Vir}$ the subalgebra of $\mathcal{S}^\prime(1,n)$ spanned by  the following elements
\[
E_k:=-\left( t^{k+1}\frac{\mathrm{d}}{\mathrm{d}t}+\left( k+1 \right)t^k \xi_1\frac{\partial}{\partial \xi_1} \right)~~\forall k\in \Z.
\]
Then $\mathrm{Vir}$ is isomorphic to the Witt algebra $\mathrm{Der}(\C[t^{\pm 1}])$.

 A simple $\mathcal{S}^\prime (1,n)$-module is called a cuspidal module if

$(1)$ $M=\oplus_{i\in \Z} M_{a+i}$, where $M_{a+i}=\{v\in M~|~ E_0 v=(a+i)v\}$ for some $a\in \C$;

$(2)$ $\dim M_{a+i}<\infty$ for all $i\in \Z$;

$(3)$ $\mathrm{supp}(M):=\{i\in \Z~|~ M_{a+i}\neq 0\}$ is neither lower bounded nor upper bounded.

Take an element
\[F:=\left( 2^n-2 \right) \xi _1\frac{\partial}{\partial \xi _1}-\sum_{i=2}^n{2^{i-1}\xi _i\frac{\partial}{\partial \xi _i}}
\in H_0.\]
Then there is an eigenspace decomposition of $\mathrm{ad}(F)$
\[\mathcal{S}^\prime(1,n)=\bigoplus_{i\in \Z} \mathcal{S}^\prime(1,n)^{(i)},\]
which provides the following triangular decomposition of $\mathcal{S}^\prime(1,n)$:
\[\mathcal{S}^\prime(1,n)=\mathcal{S}^\prime(1,n)^-\oplus \mathcal{S}^\prime(1,n)^{\left( 0 \right)}\oplus \mathcal{S}^\prime(1,n)^+,\]
where $\mathcal{S}^\prime(1,n)^{(0)}=\mathrm{Vir}+\C[t, t^{-1}]H_0$ and
$\mathcal{S}^\prime(1,n)^{\pm}=\oplus_{i>0} \mathcal{S}^\prime(1,n)^{(\pm i)}$. More precisely, $\mathcal{S}^\prime(1,n)^+$ is spanned by the following elements
\begin{align*}
  &\left\{ t^k\frac{\partial \left( \xi _I \right)}{\partial \xi _i}\frac{\mathrm{d}}{\mathrm{d}t}+\left( -1 \right) ^{1+|I|}kt^{k-1}
  \xi _I\frac{\partial}{\partial \xi _i} \,\middle|\, k\in \mathbb{Z} ,i\in \llrr[2][n],\left\{ 1 \right\} \subseteq I\subseteq \llrr \right\}
\\
\bigcup& \left\{ t^k\xi _Ih_i \,\middle|\, k\in \mathbb{Z} ,i\in\llrr[1][n-1] ,\left\{ 1 \right\}\subseteq I\subseteq \llrr \right\}
\\
\bigcup &\left\{ t^k\xi _{i_1}\xi _{i_2}\cdots \xi _{i_r}\frac{\partial}{\partial \xi _i}
\,\middle|\, k\in \mathbb{Z} ,i\in\llrr[2][n] , i>i_r>\cdots >i_1\geq 1 \right\}
.
\end{align*}
For $(\lambda,\delta,u)\in H_0^*\times \C\times \C$, let $\mathrm{Tens}\left(\lambda,\delta,u\right):=
\overline{\C[t,t^{-1}]}$  (a copy of  $\C[t,t^{-1}]$) be the $\mathcal{S}^\prime(1,n)^{(0)}$-module with actions
\begin{align*}
E_k\cdot \overline{t^m}={}&-\left(m+\delta k+u \right) \overline{t^{k+m}},\\
\left(t^kh_i\right)\cdot \overline{t^m}={}&\lambda\left(h_i\right)\overline{t^{k+m}},
\end{align*}
where $k,m\in \Z$ and $i\in\llrr[1][n-1]$.
$\mathrm{Tens}\left(\lambda,\delta,u\right)$ can be regarded as an $\mathcal{S}^\prime(1,n)^{(0)}\oplus
\mathcal{S}^\prime(1,n)^{+}$-module with a trivial action of $\mathcal{S}^\prime(1,n)^{+}$.
Let $V(\lambda,\delta,u)$
be the unique simple quotient of the generalized
Verma module
\[\mathrm{Ind}_{\mathcal{S}^\prime(1,n)^{(0)}\oplus
\mathcal{S}^\prime(1,n)^{+}}^{\mathcal{S}^\prime(1,n)} \mathrm{Tens}\left(\lambda,\delta,u\right).\]

We recall the result concerning the classification of simple
cuspidal modules over $\mathcal{S}^\prime(1,n)$ and present it in the following theorem.
\begin{theorem}[({\cite[Theorem 6 and Theorem 8]{MMZ}})]\label{t9}
  Assume that $(\lambda,\delta,u)\in H_0^*\times \C\times \C$ satisfies

  (i) $\lambda$ is dominant,

  (ii) either $\lambda(h_1)\geq 2$, or $\lambda(h_1)=1$ and $\delta =1$.
\\
Then $V(\lambda,\delta,u)$ is a simple cuspidal module over $\mathcal{S}^\prime(1,n)$.
Moreover, up to a parity-change, every simple cuspidal module over $\mathcal{S}^\prime(1,n)$ is isomorphic to one of such modules.

\end{theorem}
In the sequel, we construct the cuspidal modules over $\mathcal{S}^\prime(1,n)$ by
means of its Shen-Larsson modules. There is a $\Z$-grading of $\mathfrak{sl}(1,n)$ with respect to the adjoint action of $F$:
\begin{equation}\label{eq99}
\mathfrak{sl}(1,n)=\mathfrak{sl}(1,n)^-\oplus \mathfrak{sl}(1,n)^0\oplus \mathfrak{sl}(1,n)^+,
\end{equation}
where $\mathfrak{sl}(1,n)^0=\mathrm{span}_{\C}\{ E_{0,0}+E_{s,s}\,|\,s\in\llrr\}$ and
\begin{align*}
  \mathfrak{sl}(1,n)^-={}&\mathrm{span}_{\C}\{ E_{0,1}, E_{i,0}, E_{t,s}~|~ 1<i\leq n, 1\leq s<t\leq n\},\\
  \mathfrak{sl}(1,n)^+={}&\mathrm{span}_{\C}\{ E_{1,0}, E_{0,i}, E_{s,t}~|~ 1<i\leq n, 1\leq s<t\leq n\}.
\end{align*}
For $\mu=(\mu_1,\cdots,\mu_n)\in \C^n$,
let $\mathsf{L}(\mu)$ be the simple highest weight $\mathfrak{sl}(1,n)$-module
with the highest weight $\mu$, which is generated by the highest weight vector $v_{\mu}$ satisfying
\[
\left(E_{0,0}+E_{i,i}\right)\cdot v_\mu=\mu_i v_\mu, \forall i\in\llrr,~\textrm{and}~ \mathfrak{sl}(1,n)^+ \cdot v_\mu=0.
\]
Clearly, $\mathsf{L}(\mu)$ is finite-dimensional if and only if
$\mu$ is dominant, namely, $\mu_i-\mu_{i+1}\in \Z_+$ for all $i\in\llrr[1][n-1]$.

\begin{proposition}\label{p10}
Assume that $(\lambda,\delta,u)\in H_0^*\times \C\times \C$ satisfies

(i) $\lambda$ is dominant,

(ii) either $\lambda(h_1)\geq 2$, or $\lambda(h_1)=1$ and $\delta =1$.
\\
Then the simple cuspidal $\mathcal{S}^\prime(1,n)$-module $V(\lambda,\delta, u)$
is isomorphic to a simple subquotient of the Shen-Larsson module
\[F(t^{u-\delta}A_{1,n}, \mathsf{L}(\mu)),\]
where $\mu\in \C^n$ with $\mu_1=\delta-1$, $\mu_i=\delta-\sum_{j=1}^{i-1}\lambda(h_j)$ for $i\in\llrr[2][n]$.
In addition, $\mathsf{L}(\mu)$ is finite-dimensional.

\end{proposition}
\begin{proof}
Consider the following subspace of $F(t^{u-\delta}A_{1,n}, \mathsf{L}(\mu))$:
\[F(u-\delta,\mu):=t^{u-\delta}\C[t^{\pm 1}]\xi_1\otimes v_{\mu}=\mathrm{span}_{\C}
\left\{t^{u-\delta+m}\xi_1\otimes v_{\mu}\,\middle|\, m\in \Z\right\}.\]
For all $k,m\in\Z$ and $i\in\llrr[1][n-1]$, we have
\begin{align*}
E_k\cdot t^{u-\delta +m}\xi _1\otimes v_{\mu}={}&-\bigl( u-\delta +m+(k+1)(\mu _1+1) \bigr) t^{u-\delta +m+k}\xi _1\otimes v_{\mu}
\\
={}&-\bigl( u+m+k\delta \bigr) t^{u-\delta +m+k}\xi _1\otimes v_{\mu},\\
\bigl(t^kh_i \bigr)\cdot t^{u-\delta+m}\xi_1\otimes v_{\mu}={}&(\mu_i-\mu_{i+1}+\delta_{i,1})t^{u-\delta+m+k}\xi_1\otimes v_{\mu}\\
={}& \lambda(h_i)t^{u-\delta+m+k}\xi_1\otimes v_{\mu}.
\end{align*}
This shows that
\begin{align*}
  \phi\colon
\mathrm{Tens}\left( \lambda ,\delta ,u \right) \rightarrow &\,\,  F(u-\delta ,\mu )
\\
\overline{t^m}\mapsto &\,\, t^{u-\delta +m}\xi_1\otimes v_\mu ,\forall m\in\Z,
\end{align*}
is an $\mathcal{S}^\prime(1,n)^{\left( 0 \right)}$-module isomorphism.
Moreover, by directly computations, one can verify that $x\cdot F(u-\delta,\mu)=0$ for all $x\in \mathcal{S}^\prime(1,n)^{+}$.
Thus, by the universal property of the induced modules, there exists a unique $\mathcal{S}^\prime(1,n)$-module homomorphism
\[\varPhi \colon \mathrm{Ind}_{\mathcal{S} ^{\prime}(1,n)^{(0)}\oplus \mathcal{S} ^{\prime}(1,n)^+}^{\mathcal{S} ^{\prime}(1,n)}
\mathrm{Tens}(\lambda ,\delta ,u)\rightarrow F(t^{u-\delta}A_{1,n},\mathsf{L}(\mu ))
\]
such that $\varPhi|_{\mathrm{Tens}\left( \lambda ,\delta ,u \right)}=\phi$.
Since $V(\lambda,\delta,u)$ is the unique simple subquotient of the generalized Verma
module, it is isomorphic to a simple subquotient of
$F(t^{u-\delta}A_{1,n},\mathsf{L}(\mu))$.

The assumptions on $\lambda$ and the definition of $\mu$ ensure that
$\mu_i-\mu_{i+1}\in\mathbb{Z}_+$ for all $i\in\llrr[1][n-1]$, which implies that $\mathsf{L}(\mu)$ is
finite-dimensional.
\end{proof}

\begin{proposition}\label{t99}
  Assume that $(\lambda,\delta,u)\in H_0^*\times \C\times \C$ satisfies

(i) $\lambda$ is dominant,

(ii) either $\lambda(h_1)\geq 2$, or $\lambda(h_1)=1$ and $\delta =1$.
\\
Let $\mu\in \C^n$ with $\mu_1=\delta-1$, $\mu_i=\delta-\sum_{j=1}^{i-1}\lambda(h_j)$ for $i\in\llrr[2][n]$.
Then the following assertions hold.

(1) If $\delta=\lambda(h_1)=1$ and $\lambda(h_2)=\cdots=\lambda(h_{n-1})=0$,
then \[V(\lambda, \delta, u) \cong\begin{cases}
  \,\,\Delta A_{1,n}/\C ,&\text{for}\quad  u\in\Z,\\
  \,\,t^{u}A_{1,n} ,& \text{for}\quad u\notin \Z.
\end{cases}
\]

(2) If $\delta=\lambda(h_1)\in\Z_{\geq 2}$ and $\lambda(h_2)=\cdots=\lambda(h_{n-1})=0$,
then \[V(\lambda,\delta,u)\cong \operatorname{diff}(F(t^{u}A_{1,n},N_{n,1}((\delta-1)e_1))).\]

(3) If $\delta=\lambda(h_1)=1$, $\lambda(h_2)=\cdots=\lambda(h_{n-2})=0$ and $\lambda(h_{n-1})\in\Z_{\geq 1}$,
then \[V(\lambda,\delta,u)\cong \operatorname{diff}(F(t^{u}A_{1,n},N_{n,1}(-\mathbf{1}-\lambda(h_{n-1})e_n))).\]

(4) Except in cases (1), (2) and (3), we have
$V(\lambda,\delta,u)\cong
F(t^{u-\delta}A_{1,n},\mathsf{L}(\mu))$.
\end{proposition}
\begin{proof}
  In cases (1) and (2), the vector $t_1^{\delta-1}\in N_{n,1}((\delta-1)e_1)$ is the highest weight vector of highest weight $\mu$ with respect to the
triangular decomposition \eqref{eq99}. Hence, $\mathsf{L}(\mu) \cong N_{n,1}((\delta-1)e_1)$.
By Proposition \ref{p10}, $V(\lambda,\delta,u)$ is a simple subquotient of $F(t^{u}A_{1,n}, N_{n,1}((\delta-1)e_1))$. Then
the assertions (1) and (2) follow from Theorem \ref{t3}, Proposition \ref{p6}, \ref{p99}.

In case (3), the vector
$t_1^{-1}\cdots t_{n-1}^{-1}t_n^{-1-\lambda(h_{n-1})}\xi
\in N_{n,1}(-\mathbf{1}-\lambda(h_{n-1})e_n)$
is the highest weight vector of highest weight $\mu$ with respect to the
triangular decomposition \eqref{eq99}. Hence,
$\mathsf{L}(\mu)\cong N_{n,1}(-\mathbf{1}-\lambda(h_{n-1})e_n)$.
By Proposition \ref{p10}, $V(\lambda,\delta,u)$ is a simple subquotient of $F(t^{u}A_{1,n},N_{n,1}(-\mathbf{1}-\lambda(h_{n-1})e_n))$.
Then the assertion (3) follows from Theorem \ref{t3}.

Except in cases (1), (2) and (3), $\mathsf{L}(\mu)\ncong N_{n,1}(\mu^\prime)$ for any $\mu^\prime\in \C^n$. By Theorem \ref{t1},
we know that $F(t^{u-\delta}A_{1,n},\mathsf{L}(\mu))$
  is simple. Then assertion (4) follows from Proposition \ref{p10}.
\end{proof}

\begin{theorem}\label{t100}
 Up to a parity-change,  every simple cuspidal $\mathcal{S}^\prime(1,n)$-module is a simple subquotient of a Shen-Larsson module $F(P,M)$ for some simple weight $K_{1,n}$-module $P$ and some finite-dimensional simple $\mathfrak{sl}(1,n)$-module $M$. More precisely, up to a parity-change, it is isomorphic to one
  of the following simple cuspidal $\mathcal{S}^\prime(1,n)$-modules:

  (1) $\Delta A_{1,n}/\mathbb{C}$, $t^{a}A_{1,n}$, where
      $a\in\mathbb{C}\setminus\mathbb{Z}$;

  (2) $\operatorname{diff}\left(F\left(t^{a}A_{1,n},N_{n,1}(ke_1)\right)\right)$,
  $\operatorname{diff}\left(F\left(t^{a}A_{1,n},N_{n,1}(-\mathbf{1}-ke_n)\right)\right)$,
      where $a\in\mathbb{C}$ and $k\in \mathbb{Z}_{\geq 1}$;

  (3) $F\left(t^{a}A_{1,n},M\right)$, where $a\in\mathbb{C}$ and $M$
      is a finite-dimensional simple weight $\mathfrak{sl}(1,n)$-module that
      is not isomorphic to any $N_{n,1}(\lambda), \Pi(N_{n,1}(\lambda))$ with
      $\lambda\in \mathbb{Z}^n_+ \cup \mathbb{Z}^n_{\leq -1}$.
\end{theorem}
\begin{proof}
This follows from Theorem \ref{t9}, Proposition  \ref{p10} and Proposition \ref{t99}.
\end{proof}

\vspace*{10pt}
\noindent{\large\bf Acknowledgments}

  R. L\"{u} is partially supported by National Natural Science Foundation of China (Grant No. 12271383).
  X. Wang is supported by the China Postdoctoral
  Science Foundation (Grant No. 2025M773100) and the Postdoctoral Fellowship Program of CPSF (Grant No. GZC20252035).

\phantomsection

\end{document}